\documentclass[twoside]{article}
\usepackage[usenames,dvipsnames]{color}
\usepackage{graphicx}
\usepackage{harvard}
\usepackage{amsmath,amssymb}
    \usepackage{ulem}
    \usepackage{chngcntr}
\usepackage{mathtools}
\usepackage{bm}
\usepackage{cite}

\DeclareMathOperator*{\argmin}{arg\,min}
\def\sech{\mathrm{sech}}

\usepackage[ruled,vlined]{algorithm2e}

\usepackage{etoolbox}
\usepackage{geometry}[margins=1in]

\usepackage{tikz,pgfplots}
        \pgfplotsset{compat = 1.3}
        \pgfplotsset{minor grid style={dotted}} \pgfplotsset{major grid
        style={dashed}}

        \pgfplotsset{every x tick label/.append style={font=\footnotesize,
        yshift=0.25ex}}

        \pgfplotsset{every y tick label/.append
        style={font=\footnotesize, xshift=0.25ex}}

\definecolor{colorclassyorange}{rgb}{0.95000,0.32500,0.09800}
\definecolor{colorchromeyellow}{rgb}{1.00000,0.6549,0}%
\definecolor{colorpaleyellow}{rgb}{1.00000,0.8549,0.1}%

\definecolor{colorclassyblue}{rgb}{0.00000,0.44706,0.74118}%
\definecolor{colorpurple}{rgb}{0.49400,0.18400,0.55600}%
\definecolor{colorfuschia}{rgb}{0.95039,0.0,0.95039}%
\definecolor{colorlemongreen}{rgb}{0.6,0.8,0}%
\definecolor{colorreal}{rgb}{0.92941,0.79412,0.12549}%
\definecolor{colorimag}{rgb}{0.00000,0.49804,0.00000}%
\definecolor{colorabs}{rgb}{1.00000,0.00000,0.00000}%

\def\R#1{(\ref{#1})}

\def\MM#1{\mbox{\boldmath$#1$\unboldmath}}

\newcommand{\go}{\rightarrow}

\newcommand{\dx}{\Delta x}
\newcommand{\dt}{\Delta t}

\newcommand{\norm}[2]{\left\| #2 \right\|_{#1}}

\newcommand{\uu}{\MM{u}}

\newcommand{\xx}{\MM{x}}
\newcommand{\A}{\mathcal{A}}

    \newcommand{\DDD}{\mathcal{D}}

    \newcommand{\LLL}{\mathcal{L}}

\newcommand{\QED}{\hspace*{2em}\hfill$\Box$}

\newtheorem{theorem}{Theorem}
\newtheorem{remark}{Remark}

\newtheorem{lemma}[theorem]{Lemma}
\newenvironment{proof}[1]{\vspace{-2pt}{\em Proof\/}\quad #1}{\QED
  \vspace{8pt}}

\counterwithin{table}{section}
\counterwithin{algocf}{section}
\numberwithin{equation}{section}

\newtoggle{pgfplots}
\toggletrue{pgfplots}

\begin{document}

\title{Optimal parameters for numerical solvers of PDEs}

\author{Gianluca Frasca--Caccia \& Pranav Singh}
\maketitle

\setcounter{section}{0}

\abstract In this paper we introduce a procedure for identifying optimal methods in parametric families of numerical schemes for initial value problems in partial differential equations. The procedure maximizes accuracy by adaptively computing optimal parameters that minimize a defect-based estimate of the local error at each time-step. Viable refinements are proposed to reduce the computational overheads involved in the solution of the optimization problem, and to maintain conservation properties of the original methods. We apply the new strategy to recently introduced families of conservative schemes for the Korteweg-de Vries equation and for a nonlinear heat equation. Numerical tests demonstrate the improved efficiency of the new technique in comparison with existing methods.

\section{Introduction}
Many highly effective methods for initial value problems in partial differential equations appear as parametric families of numerical schemes. These include exponential splittings \cite{mclachlan_quispel_2002,blanes08sac}, where the free parameters constitute the coefficients of the splitting, and rational Krylov methods \cite{guettel13gm}, where the free parameters are the poles of rational approximants.

A new technique that uses symbolic algebra to develop bespoke finite difference methods that preserve multiple local conservation laws has been recently introduced in \cite{frasca}. This approach has been further refined in \cite{frasca4}, and new families of conservative schemes have been introduced for a range of partial differential equations (PDEs) in \cite{frasca,frasca2,frasca3,frasca4}. These numerical schemes feature certain free parameters that can be arbitrarily chosen without compromising the preservation of the conservation laws.

A convenient choice of the free parameters yields numerical solutions with superior accuracy in all these cases.
Coefficients of exponential splittings are typically determined a-priori using algebraic means in the pursuit of high order accuracies \cite{omf03cpc} and may be specialized for specific PDEs \cite{singh19jcp}. Optimal pole selection for rational Krylov methods remains an active area of research and strategies include a-priori choices based on analytical reasoning \cite{gg14jmaa} and a-posteriori fitting \cite{bg17rkfit}. Optimal parameters for the finite difference methods in \cite{frasca2,frasca,frasca4,frasca3} are identified using a brute force sweep through the entire parameter space, and comparisons against reference solutions show that suitable choices of the parameters yield errors up to 20 times smaller than existing methods for the proposed benchmark tests.

In practice, the optimal choice of such free parameters depends heavily on initial conditions and may also vary with time-step. Consequently, while the results in \cite{frasca,frasca2,frasca3,frasca4} do highlight the potential advantages of choosing good parameters, there is no known algorithm for identifying them. In order to overcome this issue we propose here a new approach for adaptively identifying optimal parameters for families of numerical schemes for PDEs, where convenient values are not known a priori.

For obtaining estimates of the optimal parameters we adaptively minimize an estimate of the local error introduced by the time integrator. In order for this approach to be effective, we assume throughout the paper that the spatial approximation is accurate and that the error in the solution is mainly due to the time discretization. This is not too restrictive an assumption, as in many instances PDEs are approximated very accurately in space, using for example spectral semidiscretizations. In the case of finite difference schemes, this amounts to either considering higher order discretizations in space, or restricting attention to cases where $\dt\gg\dx$. Large time-steps reduce computational expenses and are generally desirable, except for potential stability concerns. In particular, $\dt\gg\dx$ is a typical setting when using implicit schemes.

In the proposed approach, at each time-step of a single step numerical scheme, we seek to compute the optimal parameters that minimize the local error. This requires a reasonably accurate but inexpensive estimate of the local error and its dependence on the parameters. For an a posteriori estimate of the local error, we resort to the ``defect'' based approach outlined in \cite{akt13}. In the context of backward error analysis, the defect measures the discrepancy between the differential equation satisfied by the numerical solution and the original equation \cite{shampine}. Defect based error estimates have been utilized widely in the development of time-adaptive methods for ordinary differential equations (ODEs) \cite{enright89,higham89} and PDEs \cite{akt13,ahk19,ahkks19}, but, to the best of our knowledge, these have not been employed for the estimation of optimal parameters.

Unlike adaptive techniques for choosing time-steps, where the local error can be assumed to decrease monotonically with the time-step, in the proposed approach an optimization problem needs to be solved for finding the values of the parameters. The optimization problem for minimizing the local error estimate is solved in an efficient manner by computing the defect on a coarser (but still accurate) spatial grid, and utilizing an iterative method with a Gauss--Newton approximation to the Hessian for achieving locally quadratic convergence.

We apply the new procedure to families of schemes introduced in \cite{frasca} for the Korteweg de Vries (KdV) and a nonlinear heat (NLH) equation. The main feature of these schemes is that each of them preserves a specific discretization of a conservation laws. However, since the discrete conservation laws also depend on the parameters, these cannot be preserved by using an adaptive approach. Where conservation of these properties is of paramount importance, we suggest a more conservative version of the algorithm that uses fixed  parameters derived from a sequence of values obtained adaptively. 

The resulting schemes have significantly higher accuracy with moderate overheads. Despite the defect based approach being asymptotic in the time-step, $\dt$, in practice the procedure also works well in the large time-steps regime and, in some cases, also confers a notable stability advantage.

In Section~\ref{DefSec} we discuss the validity of a defect based approximation of the local error for the purpose of adaptively identifying optimal parameters. In Section~\ref{General} we outline the defect based approach used for finding optimal values of free parameters in a numerical scheme, and introduce the two algorithms briefly outlined above.  In Section~\ref{S3framework}, we apply the new techniques to families of conservative schemes introduced in \cite{frasca} for the KdV equation and the NLH equation, giving explicit expressions for the defect. In Section~\ref{numerics}, we show numerical results that demonstrate the effectiveness of the proposed algorithms in finding good estimates of the optimal parameters, together with their higher accuracy and efficiency in comparison to a default choice of the parameters and other schemes from the literature.

\section{Defect based approximation of local error}
We consider a PDE,
\begin{equation}
\label{eq:nonlinearPDE}
\partial_t u(t) = \A(u(t)), \qquad t \geq 0,\qquad u(0) = u_0 \in \mathcal{H},
\end{equation}
written as an Initial Value Problem on the Hilbert space $\mathcal{H}$, where
$\A: \mathcal{H} \go \mathcal{H}$. Boundary conditions and non-autonomous PDEs can also be incorporated into our approach in a straightforward manner, as demonstrated with concrete examples in Section~\ref{NLH}.

Following spatial discretization, the solution of \R{eq:nonlinearPDE} is approximated by the solution of the system of ODEs,
\begin{equation}
\label{eq:nonlinearODE}
\mathcal{D}_t \uu(t) = A(\uu(t)), \qquad t \geq 0,\quad \uu(0) = \uu_0 \in \mathbb{R}^M,
\end{equation}
where here and henceforth $\DDD_z$ denotes the total derivative with respect to $z$, and $\uu(t)$ represents a finite dimensional approximation of $u(t)$.
For instance, this could involve a finite difference approximation on a uniform grid on the domain $[a,b]$ with Dirichlet boundaries,
\begin{equation*}
x_m=a+m\Delta x,\qquad m=0,\ldots,M+1,\qquad \Delta x= (b-a)/(M+1).
\end{equation*}
Let $T$ be the final time of integration,
\begin{equation*}
 t_n=n\Delta t,\qquad n=0,\ldots,N,\quad \text{and} \quad \Delta t=T/N
\end{equation*}
be the time nodes and stepsize, respectively, $u_{m,n}$ an approximation of $u(x_m,t_n)$, and $\uu_n$ the column vector whose $m$-th entry is $u_{m,n}$.

The exact solution of \R{eq:nonlinearODE} is described by the flow $\mathcal{E}: \mathbb{R}^+ \times \mathbb{R}^M \go \mathbb{R}^M$,
\begin{equation*}
\uu(t) = \mathcal{E}(t, \uu_0).
\end{equation*}
Similarly, a  single step numerical scheme for \R{eq:nonlinearODE} can be described by the numerical flow,
\begin{equation*}
\uu_{n+1} = \Phi(\Delta t, \uu_n).
\end{equation*}
Note that the numerical flow $\Phi$ also exists for implicit methods, even if not specified in an explicit form.

In this manuscript, we consider numerical schemes in the form
\begin{equation}
\label{eq:solver}
\uu_{n+1} =  \Phi(\dt, \uu_n, \chi), \qquad \Phi: \mathbb{R}^+ \times \mathbb{R}^M \times \Omega \go \mathbb{R}^M,
\end{equation}
where $\Omega$ is a compact subset of $\mathbb{R}^K$, and $\Phi$ depends on a vector of free parameters $\chi \in \Omega$ that effect the accuracy of the scheme. In our theoretical discussion we assume that the vector field $A$, the exact flow $\mathcal{E}$ and the numerical flow $\Phi$ are smooth with respect to all arguments.

The local error in the numerical method  \R{eq:solver} is defined as
\begin{equation}
\label{eq:LEe}
\LLL(\dt,\uu_n,\chi) = \Phi(\dt, \uu_n,  \chi) - \mathcal{E}(\dt, \uu_n).
\end{equation}
In general, $\LLL(\dt,\uu_n,\chi)$ is not a computable quantity since the exact solution $\mathcal{E}(\dt, \uu_n)$ is not available in practice. Consequently, we resort to defect-based approximations \cite{akt-linear,akt13} to obtain a posteriori estimates.
The {\em defect} or {\em residual} of $\Phi$,
\begin{equation}
\label{eq:defect}
\mathcal{R}(\dt, \uu_n, \chi) = \mathcal{D}_{\dt} \Phi(\dt, \uu_n, \chi) - A(\Phi(\dt, \uu_n,  \chi)),
\end{equation}
quantifies the extent to which the numerical flow $\Phi$ fails to satisfy \R{eq:nonlinearODE}.

For nonlinear parabolic PDEs and time-reversible equations, the following nonlinear variation-of-constant formula holds true \cite{ahk19,DescThal}. Henceforth, $\partial_k f$ denotes the Fr\'echet derivative of a function $f$ with respect to the $k$-th argument.
\begin{lemma}[Gr\"obner-Alekseev formula]\label{GAlemma}
The analytical solutions of the following initial value problems
\begin{align*}
&\left\{\begin{array}{l}
\mathcal{D}_t\uu(t)=H(t,\uu(t))=G(\uu(t))+R(t,\uu(t)),\qquad 0\leq t \leq T,\\
\uu(0)=\uu_0,
\end{array}\right.\\
&\left\{\begin{array}{l}
\mathcal{D}_t\uu(t)=G(\uu(t)),\qquad 0\leq t \leq T,\\
\uu(0)=\uu_0,
\end{array}\right.
\end{align*}
are related through the nonlinear variation-of-constants formula
$$\mathcal{E}_H(t,\uu_0)=\mathcal{E}_G(t,\uu_0)+\int_0^t\partial_2\mathcal{E}_G(t-\tau,\mathcal{E}_H(\tau,\uu_0))\cdot R(\tau,\mathcal{E}_H(\tau,\uu_0))\mathrm{d}\tau, \quad 0\leq t\leq T.$$
\end{lemma}
Applying Lemma \ref{GAlemma} to (\ref{eq:nonlinearODE}) and (\ref{eq:defect}) yields the following formula for the local error (\ref{eq:LEe}),
\begin{align}
\label{eq:LEerror}
\LLL(\dt, \uu_n,\chi) &= \int_{0}^{\dt} \partial_2 \mathcal{E}\left(\dt - \tau, {\Phi}(\tau, \uu_n, \chi) \right) \cdot \mathcal{R}(\tau, \uu_n, \chi)\, \mathrm{d}\tau =: \int_0^{\Delta t} \Theta(\tau,\uu_n,\chi)\,\mathrm{d}\tau.
\end{align}
The following theorem shows that if (\ref{eq:solver}) is a method of order $p$, then the quantity
\begin{equation}\label{estim}
L(\dt,\uu_n,\chi)=\frac{\dt}{p+1}\mathcal{R}(\dt,\uu_n,\chi),
\end{equation}
is an asymptotically correct estimator of the local truncation error (\ref{eq:LEe}), uniformly for $\chi\in\Omega$.
\begin{theorem}\label{defproof}
The error estimate (\ref{estim}) is such that
$$\|L(\dt,\uu_n,\chi)-\LLL(\dt,\uu_n,\chi)\|\leq C\dt^{p+2},$$
uniformly for any value $\chi\in\Omega,$ i.e. with $C$ independent of $\chi$.
\end{theorem}
\begin{proof}
Since the method is of order $p$, then $\LLL(\dt,\uu_n,\chi)=\mathcal{O}(\dt^{p+1}),$ and $\Theta(\tau,\uu_n,\chi)=\mathcal{O}(\tau^{p})$ for any $\Delta t,$ $\tau\in[0,\dt]$, and $\chi\in\Omega.$
Taylor expanding $\Theta$ in $\tau$ around $0$, there exists $\xi_1 \in [0,\tau]$ such that
\begin{equation}
    \label{eq:TaylorTheta}
        \Theta(\tau,\uu_n,\chi) = \frac{\tau^p}{p!}\partial_1^p\Theta(0,\uu_n,\chi)+\frac{\tau^{p+1}}{(p+1)!}\partial_1^{p+1}\Theta(\xi_1,\uu_n,\chi).
\end{equation}
Therefore,
\begin{align}
\nonumber \LLL(\dt,\uu_n,\chi)&=\int_0^{\Delta t}\Theta(\tau,\uu_n,\chi)\,\mathrm{d}\tau \\
\label{eq:LLLexpansion}& = \frac{\dt^{p+1}}{(p+1)!}\partial_1^p\Theta(0,\uu_n,\chi) + \int_0^{\Delta t}\frac{\tau^{p+1}}{(p+1)!}\partial_1^{p+1}\Theta(\xi_1,\uu_n,\chi)\,\mathrm{d}\tau \, .
\end{align}
Due to smoothness of $\Theta$, $\partial_1^{p+1}\Theta$ is continuous and is bounded over the compact set $[0,\Delta t] \times \Omega$, so that we may define
$$\widetilde{C}=\max_{(\xi,\chi)\in[0,\dt]\times\Omega}\left\|\partial_1^{p+1}\Theta(\xi,\uu_n,\chi) \right\| < \infty.$$

Note that by definition of $\Theta$ in (\ref{eq:LEerror}), $\Theta(\dt,\uu_n,\chi) = \partial_2 \mathcal{E}\left(0, {\Phi}(\dt, \uu_n, \chi) \right) \cdot \mathcal{R}(\dt, \uu_n, \chi) = \mathcal{R}(\dt,\uu_n,\chi)$.
Using (\ref{eq:LLLexpansion}) and applying (\ref{eq:TaylorTheta}) with $\tau = \Delta t$, for some $\xi_1,\xi_2 \in [0,\Delta t]$,
\begin{align*}
& \left\|\LLL(\dt,\uu_n,\chi) - \frac{\dt}{p+1}\mathcal{R}(\dt,\uu_n,\chi)\right\|  = \left\|\LLL(\dt,\uu_n,\chi) - \frac{\dt}{p+1}\Theta(\dt,\uu_n,\chi)\right\|\\
& \qquad \qquad =  \left\|
\int_0^{\Delta t}\frac{\tau^{p+1}}{(p+1)!}\partial_1^{p+1}\Theta(\xi_1,\uu_n,\chi)\,\mathrm{d}\tau
-\frac{\Delta t^{p+2}}{(p+1)!(p+1)}\partial_1^{p+1}\Theta(\xi_2,\uu_n,\chi)
\right\|\\
& \qquad \qquad \leq  \left\|
\int_0^{\Delta t}\frac{\tau^{p+1}}{(p+1)!} \partial_1^{p+1}\Theta(\xi_1,\uu_n,\chi)\,\mathrm{d}\tau \right\| +
\left\|\frac{\Delta t^{p+2}}{(p+1)!(p+1)}\partial_1^{p+1}\Theta(\xi_2,\uu_n,\chi)
\right\|\\
& \qquad \qquad \leq \left[\frac{1}{(p+2)!}+\frac{1}{(p+1)!(p+1)}\right]\widetilde{C}\dt^{p+2} =: C\dt^{p+2},
\end{align*}
where $$C=\frac{2p+3}{(p+2)!(p+1)}\widetilde{C},$$
is independent of $\dt$ and $\chi$. Therefore, indeed
$$\|L(\dt,\uu_n,\chi)-\LLL(\dt,\uu_n,\chi)\|\leq C\dt^{p+2},$$
uniformly for any value of $\chi\in\Omega$.
\end{proof}

\label{DefSec}

\section{Defect based identification of optimal parameters}\label{General}
In this section, we propose the use of the defect based error estimate (\ref{estim}) for finding, at every time-step, optimal parameters $\chi_n^* \in \Omega \subset \mathbb{R}^K$, defined as
\begin{equation}
\label{eq:opt}
\chi_n^* = \argmin_{\chi \in \Omega} \norm{2}{L(\dt,\uu_n,\chi)},
\end{equation}
where $\uu_n$ and $\dt$ are fixed. The result of Theorem \ref{defproof} assures us that
 $$\mathcal{L}(\dt,\uu_n,\chi_n^*) = L(\dt,\uu_n,\chi_n^*) + C{\dt^{p+2}},$$ and guarantees that the choice of parameters $\chi_n^*$ keeps the true local error, $\mathcal{L}$, close to $L(\dt,\uu_n,\chi_n^*)$ in the asymptotic limit $\dt \go 0$ and, therefore, small.

\begin{remark}\label{parvstime}
The application of defect based error estimates for choosing optimal parameters differs from their application in context of time-adaptivity in a couple of crucial aspects.
\begin{enumerate}
\item Since $L$ is neither monotonous in $\chi$, nor are we interested in asymptotic limits for small $\chi$ (unlike the case of $\dt$ in context of time-adaptivity), the defect based estimate $L(\dt,\uu_n,\chi)$ needs to be computed for multiple values of $\chi$ within an optimization routine.
\item The perturbation of $L$ by a term $\rho$ independent of $\chi$ has no effect on $\chi^*$. This is in contrast to time-adaptivity where we seek the largest $\dt^*$ such that $\norm{2}{L(\dt^*,\uu_n,\chi)} < \delta$ for some user specified tolerance $\delta$, and $\rho\neq 0$ effects the choice of $\dt^*$.
\end{enumerate}
\end{remark}
The first observation in Remark~\ref{parvstime} suggests that the application of defect based estimates for choosing optimal parameters can be prohibitively expensive. However, the second suggests that we can resort to inexpensive approximations of the defect and still hope to arrive at a good choice of parameters.

In Section~\ref{sec:efficiency}, we see that under reasonable assumptions, the number of optimization steps is not expected to be large and just a few steps of Gauss--Newton iteration are ever required. In the large time-step regime the approximation of defect on a coarse spatial grid also proves to be sufficient for the purposes described here. Overall, this leads to a procedure for identifying optimal parameters with very reasonable overheads, producing highly efficient schemes.

\subsection{Optimization procedure}
\label{sec:opt}

In practice, we minimize the square of the defect,
\begin{equation*}
\chi^*_n = \argmin_{\chi \in \Omega} f(\chi), \qquad f(\chi) = \frac12 \norm{2}{\mathcal{R}(\dt,\uu_n,\chi)}^2,
\end{equation*}
using the gradient,
\begin{equation}\label{eq:grad}
    g = \left(\mathcal{D}_\chi \mathcal{R}(\dt,\uu_n,\chi)\right)^\top \mathcal{R}(\dt,\uu_n,\chi),
\end{equation}
where $\mathcal{D}_\chi \mathcal{R}$ is the Jacobian of the defect with respect to $\chi$,
and a Gauss--Newton approximation to the Hessian,
\begin{equation}\label{eq:hess}
    \nabla^2 f \approx H_{\mathrm{GN}} :=  \left(\mathcal{D}_\chi \mathcal{R}(\dt,\uu_n,\chi)\right)^\top \left(\mathcal{D}_\chi \mathcal{R}(\dt,\uu_n,\chi)\right).
\end{equation}
Utilizing the Gauss--Newton Hessian in the context of trust region algorithms yields a sequence of parameters $\chi_n^k$ that quadratically converge to the optimal $\chi^*_n$, with reliable global convergence properties \cite{NoceWrig06}. At the same time, the procedure remains relatively inexpensive for a small number of parameters, $K$, since we only need to compute the first derivatives of the defect with regards to $\chi$. These can be computed either analytically or approximately using finite differences.

The defect, $\mathcal{R}(\dt, \uu_n, \chi_n^k)$, is computed using \R{eq:defect}. This requires the computation of a temporary solution, $\widetilde{\uu}_{n+1}^k = \Phi(\dt, \uu_n, \chi_n^k)$, and of $\mathcal D_{\dt} \Phi(\dt,\uu_n, \chi_n^k)$, the latter of which can be computed analytically as outlined with concrete examples in Section \ref{S3framework}.

Note that, in general, at each iteration a trust region algorithm may be used to compute $\mathcal{R}$ at candidate parameters $\chi=\widetilde{\chi}_n^{k+1}$ before deciding to accept or reject the candidate and/or update the trust region radius $\Delta_n^k$. For a detailed introduction to trust region algorithms, we refer the reader to \cite{NoceWrig06}.

\subsection{Practical considerations for efficiency}\label{sec:efficiency}

The evaluations of defect can be very expensive, as they require the computation of the temporary solution $\widetilde{\uu}_{n+1}^k$ at every iteration. Each of these is as expensive as a step of the original numerical solver $\Phi$.
However, in practice we identify $\chi_n^*$ by optimizing the defect on a coarse spatial grid, resorting to the fine computational grid only for evaluating $\uu_{n+1}$ once $\chi_n^*$ has been identified.

The coarse grid is obtained as a subgrid of the fine grid with resolution $r \dx$, with $r$ an integer that divides $M+1$. Let $\mathcal{P}_r$ denote the projection operator from the fine grid to the coarse grid.
At the $k$-th iteration of Gauss-Newton algorithm, we evaluate the defect (\ref{eq:defect}) on the coarse grid, as
$$\mathcal{R}(\dt, \widehat{\uu}_n,\chi_n^k)=\mathcal{D}_{\dt} \Phi(\dt, \widehat{\uu}_n, \chi_n^k) - A(\Phi(\dt, \widehat{\uu}_n,  \chi_n^k)),\qquad \widehat{\uu}_n=\mathcal{P}_r(\uu_n).$$ This requires the computation of $\widetilde{\uu}_{n+1}^k = \Phi(\dt, \widehat{\uu}_n, \chi_n^k)$. On a grid with resolution $r\dx$, the dimension of the problem is reduced by a factor $r$. This typically leads to a significant speedup in the computation of $\chi_n^*$.
This speedup is expected to be particularly pronounced in 2 or 3 dimensional problems, where the coarse grid is smaller by factors of $r^2$ and $r^3$, respectively, than the finer computational grid.

For a method of order $p$ in space and time, a $\mathcal{O}(r^p\dx^p)$ term of error is introduced in the evaluation of the defect. This error is negligible if $r\dx \ll \dt$ and is not expected to have a significant effect on the estimate of the optimal parameters $\chi_n^*$ in light of Remark~\ref{parvstime}.

\begin{remark}\label{remK}
With larger values of $r$, the additional cost of identifying $\chi_n^*$ becomes marginal, while the advantages of identifying good parameters could still be significant.
We can expect, however, that once $r$ is large enough such that $r\dx \ll \dt$ is no longer valid, spatial discretization errors will start to dominate and the computation of defect may become too inaccurate to be useful. In light of these observations, as a rule of thumb, the largest $r$ we recommend is the largest divisor of $M+1$ such that $r<\dt/\dx$.
\end{remark}

Further gains can be obtained by exploiting temporal smoothness of the optimal parameters $\chi^*$. If the solution $\uu(t)$ is smooth in time, it is reasonable to assume that the optimal parameters are described by a Lipschitz function $\chi^*(t)$, i.e. $|\chi^*_{n} - \chi^*_{n-1}| \leq \widetilde{C} \dt$ for some $\widetilde{C}<\infty$ independent of $n$. For small enough $\widetilde{C}\dt$, $\chi^*_{n-1}$ is close enough to $\chi^*_{n}$. Thus, the previous value of the optimal parameter serves as a good first guess for the next time-step, $\chi^{0}_{n} = \chi^*_{n-1}$. With $\widetilde{C}$ and $\dt$ small enough, $\chi^{0}_{n}$ can be expected close enough to $\chi^{*}_{n}$ so that conditions of trust-region are guaranteed to be satisfied and quadratic convergence is guaranteed. In such a case, it suffices to use the simple Newton-type iteration,
\begin{equation}
\label{eq:iter}
 \chi^{k+1}_{n} = \chi^{k}_{n} - H_{\mathrm{GN}}^{-1}\ g, \qquad \chi^0_n = \chi^*_{n-1}, \quad n>0,
\end{equation}
in place of the trust-region algorithm, where $g$ and $H_{\mathrm{GN}}$ are given by \R{eq:grad} and \R{eq:hess}, respectively. In practice, just a couple of steps of \R{eq:iter} suffice, with the exception of the first iteration ($n=0$), when the arbitrary value of the initial guess may be very far from the optimal one, requiring more optimization steps and the use of trust-regions.

Gauss-Newton algorithm is iterated until the stopping criterion,
$$|\chi_n^{k+1}-\chi_n^{k}|<{\rm tol},$$
is satisfied for a suitable tolerance. Then we set ${\chi}^*_n = \chi_n^{k+1}$ and the solution at the next time-step is obtained on the finer computational grid as
$${\uu}_{n+1} = \Phi(\dt, {\uu}_n, {\chi}^*_n).$$
The overall procedure introduced in this section is summarized in Algorithm~\ref{General}.\ref{alg:opt}.

\begin{minipage}{.9\linewidth}
\begin{algorithm}[H]
\label{alg:opt}
\SetAlgoLined
\SetKwInOut{Input}{input}
\SetKwInOut{Output}{output}
\Output{$\uu_0,\uu_1,\ldots,\uu_N$, $\chi_0^*, \chi_1^*, \ldots, \chi_{N-1}^*$.}
\Input{$\chi_0^0 \in \Omega, \uu_0, \Phi, \dt, N, \Delta^0_0, \mathrm{MaxIter}, \mathrm{tol}_1, \mathrm{tol}_2, r$.}
\For{$n \gets 0$ \KwTo $N-1$}
    {
    $\widehat{\uu}_n=\mathcal{P}_r(\uu_n)$\;
    $k\gets 0$\;
        \While{$k \leq \mathrm{MaxIter}$  {\bf and} $|\chi_n^{k+1}-\chi_n^k|> \mathrm{tol}$ }
        {
        $\mathcal{R} \gets \mathcal{R}(\dt, \widehat{\uu}_n, \chi_n^k)$\;
        $\mathcal{D}_\chi \mathcal{R}\gets \mathcal{D}_\chi \mathcal{R}(\dt, \widehat{\uu}_n, \chi_n^k)$\;
        $g \gets \left(\mathcal{D}_\chi \mathcal{R}\right)^\top \mathcal{R}$\;
        $H_{\mathrm{GN}} \gets \left(\mathcal{D}_\chi \mathcal{R}\right)^\top \left(\mathcal{D}_\chi \mathcal{R}\right)$\;
        $\chi_n^{k+1}, \Delta_n^{k+1} \gets {\bf TrustRegion}(\chi_n^k, \Delta_n^k, g, H_{\mathrm{GN}}, \widehat{\uu}_n, \Phi)$\;
        $k\gets k+1$\;
        }
    $\chi_n^* \gets \chi_n^k$\;
    $\uu_{n+1} \gets \Phi(\dt, \uu_n, \chi_n^*)$\;
    $\chi_{n+1}^0 \gets \chi_n^*$\;
    $\Delta_{n+1}^0 \gets \Delta_n^{k}$\;
    }
 \caption{Adaptively identifying optimal parameters}
\end{algorithm}
\end{minipage}

\subsection{Modifications for conservative schemes}\label{sec:def4cons}
Algorithm~\ref{General}.\ref{alg:opt} is very effective in improving accuracy of standard numerical methods with free parameters and of geometric integrators that preserve a structure that is independent of the parameters. However, the schemes described in \cite{frasca} preserve some conservation laws that depend on the parameters. Therefore, a time adaptive approach undermines the preservation of these conservation laws.

In order to obtain numerical solutions that satisfy the physical constraints given by these conservation laws, we modify Algorithm~\ref{General}.\ref{alg:opt} as follows.
\begin{enumerate}
\item Project the initial condition on a grid with resolution $r\dx$.
\item Find the optimal $\chi_0^*$ and use it to advance a step in time on the coarse grid. Iterate this step till the final time, obtaining the optimal parameters $\chi_n^*$ at $n=0,\ldots,N-1$ steps.
\item Compute the average optimal parameters as $\overline{\chi}^* = \frac{1}{N} \sum_{n=0}^{N-1} \chi_n^*$.
\item Compute the numerical solution on the full computational grid using the average parameters, $\uu_{n+1} = \Phi(\dt, \uu_n, \overline{\chi}^*)$, for $n=0,\ldots,N-1$.
\end{enumerate}
We summarize this conservative version of our procedure in Algorithm~\ref{General}.\ref{alg:opt2}.
Note that the numerical solution obtained on the coarse grid, $\widehat{\uu}_n$, is used only for estimating the optimal parameters and later discarded.

\begin{minipage}{.9\linewidth}
\begin{algorithm}[H]
\label{alg:opt2}
\SetAlgoLined
\SetKwInOut{Input}{input}
\SetKwInOut{Output}{output}
\Output{$\uu_0,\uu_1,\ldots,\uu_N$, $\chi_0^*, \chi_1^*, \ldots, \chi_{N-1}^*$.}
\Input{$\chi_0^0 \in \Omega, \uu_0, \Phi, \dt, N, \Delta^0_0, \mathrm{MaxIter}, \mathrm{tol}_1, \mathrm{tol}_2, r$.}
$\widehat{\uu}_0\gets \mathcal{P}_r(\uu_0)$\;
\For{$n \gets 0$ \KwTo $N-1$}
    {
    $k\gets 0$\;
        \While{$k \leq \mathrm{MaxIter}$  {\bf and}  $|\chi_n^{k+1}-\chi_n^k|> \mathrm{tol}$ }
        {
        $\mathcal{R} \gets \mathcal{R}(\dt, \widehat{\uu}_n, \chi_n^k)$\;
        $\mathcal{D}_\chi \mathcal{R}\gets \mathcal{D}_\chi \mathcal{R}(\dt, \widehat{\uu}_n, \chi_n^k)$\;
        $g \gets \left(\mathcal{D}_\chi \mathcal{R}\right)^\top \mathcal{R}$\;
        $H_{\mathrm{GN}} \gets \left(\mathcal{D}_\chi \mathcal{R}\right)^\top \left(\mathcal{D}_\chi \mathcal{R}\right)$\;
        $\chi_n^{k+1}, \Delta_n^{k+1} \gets {\bf TrustRegion}(\chi_n^k, \Delta_n^k, g, H_{\mathrm{GN}}, \widehat{\uu}_n, \Phi)$\;
        $k\gets k+1$\;
        }
    $\chi_n^* \gets \chi_n^k$\;
    $\widehat{\uu}_{n+1} \gets \Phi(\dt, \widehat{\uu}_n, \chi_n^*)$\;
    $\chi_{n+1}^0 \gets \chi_n^*$\;
    $\Delta_{n+1}^0 \gets \Delta_n^{k}$\;
    }
$\overline{\chi}^* \gets \frac{1}{N} \sum_{n=0}^{N-1} \chi_n^*$\;
\For{$n \gets 0$ \KwTo $N-1$}
    {
    $\uu_{n+1} \gets \Phi(\dt, \uu_n, \overline{\chi}^*)$\;
    }
 \caption{Identifying fixed optimal parameters}
\end{algorithm}
\end{minipage}
\vspace{10pt}

Since our estimate of the optimal parameters in Algorithm~\ref{General}.\ref{alg:opt2} relies on the solution of the problem on a coarser grid, this algorithm is also affected by the accumulation of errors in space, which may be particularly pronounced for large values of $r$. On the other hand, the parameters $\chi_n^*$ need not be identified with too high an accuracy since we are only interested in a good average choice, $\overline{\chi}^*$.

Considering the average parameters for solving the problem on the fine grid is a good option particularly for solutions with simple and smooth dynamics (e.g. travelling waves) as the values obtained at the different time-step are all reasonably close to each other.

\begin{remark}
Note that at the end of step 2, the full sequence of optimal parameters, $\chi_0^*, \ldots, \chi_{N-1}^*$, and local error estimates are available to the user, and it is possible to consider alternative options than the average to derive different single fixed values of the parameters.
\end{remark}

\section{Approximation of defect for specific schemes}\label{S3framework}
In this section we
consider the application of the proposed approach to two partial differential equations --
the KdV equation and a nonlinear heat equation -- with suitable initial and boundary conditions.
In particular, we present the computation of defect (\ref{eq:defect})
for numerical schemes introduced in \cite{frasca}, which is required
in Algorithms~\ref{General}.\ref{alg:opt} and \ref{General}.\ref{alg:opt2}.

We define the forward shifts in space and time,
$$S_{\dx}(u_{m,n})=u_{m+1,n},\qquad S_{\dt}(u_{m,n})=u_{m,n+1},$$
respectively, the forward difference and average operators,
$${D}_{\Delta x}=\frac{S_{\dx}-I}{\Delta x},\qquad {D}_{\Delta t}=\frac{S_{\dt}-I}{\Delta t},\qquad \mu_{\Delta x}=\frac{S_{\dx}+I}{2},\qquad \mu_{\Delta t}=\frac{S_{\dt}+I}{2},$$
and the centred difference operator,
$$D_{2k,\Delta x}=D_{\dx}^{2k}S_{\dx}^{-k},\qquad D_{2k-1,\Delta x}=D_{\dx}^{2k-1}S_{\dx}^{-k}\mu_{\dx},$$ approximating the space derivatives of degree $2k$ and $2k-1$, respectively, with second order of accuracy. Action of these operators on vectors is defined entrywise. Moreover, we denote with $\circ$ the Hadamard product whose action is entrywise multiplication of vectors.

For the two equations considered here, families of second order finite difference methods that depend on one or more free parameters have been introduced in \cite{frasca}. We consider in this section some of these families of schemes.
\begin{remark}\label{parsize}
We restrict attention to the setting where $\dt>\dx$, and the parameters $\chi$ featured in the numerical schemes from \cite{frasca} are $\mathcal{O}(\Delta t^2)$. These small parameters $\chi$ correspond to perturbation terms that have no counterpart in the continuous problem, and whose contributions vanish in the limit $\Delta t \go 0$. Thus, it is reasonable to restrict our search to a neighbourhood of $\mathbf{0}$ of size $\mathcal{O}(\Delta t^2)$, $\Omega\subset \overline{B}_{C\Delta t^2}(\mathbf{0};\mathbb{R}^K)$, for some constant $C$.  Therefore, when applying the algorithms outlined in Section~\ref{General} we can expect that $\chi=\mathbf{0}$ is a reasonable initial guess, and we can find convenient values of the parameters by using the simple iteration (\ref{eq:iter}) without the aid of trust region methods.
\end{remark}
\begin{remark}\label{Jacobrem}
The schemes considered here are all implicit. We assume that $\mathcal{J}$, the Jacobian operator defined by the partial derivatives of the scheme with respect to $\uu_{n+1}$, is never singular. Solutions can then be obtained by iteration of Newton's method.
\end{remark}
An important property of the numerical schemes considered here is that they preserve some conservation laws. Conservation laws are defined as total divergences,
$$\DDD_x  F + \DDD_t G,$$
that vanish on solutions of the PDE. The functions $F$ and $G$ are the flux and the density of the conservation law and depend on $x,t,u$ and its derivatives. The methods in \cite{frasca} preserve second order approximations of specific conservation laws in the form
$$D_{\dx}  \widetilde F(\xx, \uu_n, \uu_{n+1},\chi) + D_{\dt} \widetilde G(\xx, \uu_n, \chi)=0,$$
where here and henceforth tildes represent approximations of the corresponding continuous terms, and $\xx$ is the column vector whose $m$-th entry is $x_m$.

\subsection{Schemes for KdV equation}\label{KdV1}
In this section we apply the approach outlined in Section~\ref{General} to parametric families of schemes for the KdV equation,
\begin{equation}
\label{eq:nonlinearPDEex2}
u_t + \left(\tfrac{1}2u^2+u_{xx}\right)_x=0,
\end{equation}
with initial condition $$u(x,0)=u_0(x).$$ For simplicity, we restrict attention to periodic or zero boundary conditions. However, as shown in Section~\ref{NLH}, the entire discussion can also be adapted to boundary conditions of a different type.
The KdV equation has infinitely many independent conservation laws. The first three, in increasing order,
\begin{align}\label{eq:massCL}
&\DDD_xF_1+\DDD_tG_1\equiv\DDD_x\left(\tfrac{1}2u^2+u_{xx}\right)+\DDD_tu=0,\\\nonumber
&\DDD_xF_2+\DDD_tG_2\equiv\DDD_x\left(\tfrac{1}3u^3+uu_{xx}-\tfrac{1}2u_x^2\right)+\DDD_t\left(\tfrac{1}2u^2\right)=0,\\\nonumber
&\DDD_xF_3+\DDD_tG_3\equiv\DDD_x\left(\tfrac{1}4u^4+u_xu_t-uu_{xt}+u^2u_{xx}+u^2_{xx}\right)+\DDD_t\left(\tfrac{1}3u^3+uu_{xx}\right)=0,
\end{align}
describe the local conservation laws of mass, momentum and energy, respectively.
For this equation we define the semidiscrete operator $A$ in (\ref{eq:nonlinearODE}) in the most natural way, as
\begin{equation}\label{eq:Aop}
A(\uu(t))=-D_{1,\dx}\left(\tfrac{1}{2}\uu^2(t)+D_{2,\dx}\uu(t)\right).
\end{equation}
The results obtained are independent of the particular form of this operator under spatial discretization, as we work under the assumption that the leading source of error is given by the time integration.

\subsubsection*{Energy conserving methods}
We consider here the following family of mass and energy-conserving methods described in \cite{frasca},
\begin{equation}\label{ECschemes}
\text{EC}(\alpha)\equiv D_{\dt}(\uu_n)+D_{\dx}\left(S_{\dx}^{-1}\mu_{\dx}\psi(\uu_n,\uu_{n+1},\alpha)\right)=0,
\end{equation}
where
$$\psi(\uu_n,\uu_{n+1},\alpha)=\tfrac{1}6(\uu_{n+1}^2+\uu_n^2+\uu_n\circ\uu_{n+1})+D_{2,\dx}\mu_{\dt}\uu_n+\alpha D_{\dt}D_{1,\dx}\uu_n.$$
The schemes EC$(\alpha)$ have a discrete version of the conservation law of the mass (\ref{eq:massCL}) given by their definition (\ref{ECschemes}) with,
\begin{equation*}
\widetilde F_1=S_{\dx}^{-1}\mu_{\dx}\psi(\uu_n,\uu_{n+1},\alpha),\qquad \widetilde G_1=\uu_n.
\end{equation*}
Moreover, they possess the discrete energy conservation law,
\begin{align}\label{EnergyEC}
D_{\dx}&\big(\widetilde F_3\big)+D_{\dt}\big(\widetilde G_3\big)=0,\\\nonumber
\widetilde F_3=&\,\psi(\uu_n,\uu_{n+1},\alpha)\!\circ\! S_{\dx}^{-1}\psi(\uu_n,\uu_{n+1},\alpha)+\alpha (D_{\dt}\uu_n)\circ (S_{\dx}^{-1}D_{\dt}\uu_n)\\\nonumber
&+S_{\dx}^{-1}\big((D_{\dx}\mu_{\dt}\uu_n)\circ (D_{\dt}\mu_{\dx}\uu_n)-(\mu_{\dx}\mu_{\dt}\uu_n)\circ (D_{\dx}D_{\dt}\uu_n)\big),\\\nonumber
\widetilde G_3=&\,\tfrac{1}3\uu_n^3+\uu_n\circ D_{2,\dx}\uu_n.
\end{align}
Being implicit methods, an explicit expression for $\Phi$ in \R{eq:solver} is not available. Nevertheless, an analytical expression of its time derivatives can be obtained by substituting \R{eq:solver} in (\ref{ECschemes}) and differentiating. This yields,
\begin{align}\label{eq:DtEC}
\DDD_{\dt} \Phi(\dt,\uu_n,\alpha) = -[\dt\mathcal{J}]^{-1}(D_{1,\dx}\psi(\uu_n,\uu_{n+1},0)),
\end{align}
where $\uu_{n+1}=\Phi(\dt,\uu_n,\alpha)$,  and $\mathcal{J}$ denotes the Jacobian matrix defined in Remark \ref{Jacobrem}.

Optimal values of $\alpha$ are then computed according to \R{eq:opt}, where the defect is calculated according to (\ref{eq:defect}) with (\ref{eq:Aop}) and (\ref{eq:DtEC}).

\begin{remark}\label{remEC}
As noted in Section~\ref{sec:def4cons}, since the value of $\alpha$ changes at every time-step in Algorithm~\ref{General}.\ref{alg:opt}, it cannot preserve the conservation laws (\ref{ECschemes}) and (\ref{EnergyEC}) of EC($\alpha$) because they depend on $\alpha$. However, since the boundary conditions are conservative, summing the entries of the vectors in (\ref{ECschemes}) and (\ref{EnergyEC}) yields
$$\sum_m u_{m,n+1}=\sum_m u_{m,n},\quad \sum_m \left(\tfrac{1}3 u_{m,n+1}^3+u_{m,n+1} D_{2,\dx}u_{m,n+1}\right)=\sum_m \left(\tfrac{1}3 u_{m,n}^3+u_{m,n} D_{2,\dx}u_{m,n}\right).$$
Therefore, EC($\alpha$) also preserves the following approximations of the global mass and energy:
$$\dx\sum_m u_{m,n},\qquad \dx\sum_m \left(\tfrac{1}3 u_{m,n}^3+u_{m,n} D_{2,\dx}u_{m,n}\right).$$ These two global invariants are independent of $\alpha$, and therefore they are conserved by both algorithms introduced in Section~\ref{General}.
\end{remark}
\subsubsection*{Momentum conserving methods}
A two-parameter family of mass and momentum conserving schemes described in \cite{frasca} is 
\begin{align}\nonumber
\text{MC}(\beta,\gamma)\equiv&\, D_{\dt}(\uu_n)+D_{\dx}\left\{\tfrac{1}6\big((S_{\dx}^{-1}\mu_{\dt}\uu_n)^2+(\mu_{\dt}\uu_n)^2+(S_{\dx}^{-1}\mu_{\dt}\uu_n)\circ(\mu_{\dt}\uu_n)\big)\right.\\\label{MCschemes}
&\left.+D_{\dx}^2S_{\dx}^{-2}\mu_{\dx}\mu_{\dt}\uu_n+D_{\dt}D_{\dx}S_{\dx}^{-1}(\beta\uu_n+\gamma D_{2,\dx}\uu_n)\right\}=0.
\end{align}
Solutions of MC$(\beta,\gamma)$ satisfy the local mass conservation law given by (\ref{MCschemes}), and the local momentum conservation law,
\begin{align}\label{MomMC}
D_{\dx}&\big(\widetilde F_2\big)+D_{\dt}\big(\widetilde G_2\big)=0,\\\nonumber
\widetilde F_2=&\,\tfrac{1}3(\mu_{\dt}\uu_n)\circ(S_{\dx}^{-1}\mu_{\dt}\uu_n)\circ(S_{\dx}^{-1}\mu_{\dt}\mu_{\dx}\uu_n)+\tfrac{1}2(\mu_{\dt}\uu_n)\circ(D_{2,\dx}\mu_{\dt}\uu_n)\\\nonumber
&+\tfrac{1}2(S_{\dx}^{-1}\mu_{\dt}\uu_n)\circ(D_{\dx}^2\mu_{\dt}\uu_n)-\tfrac{1}2(D_{1,\dx}\uu_n)^2+\beta\rho(\uu_n,\uu_{n+1}) +\gamma\sigma(\uu_n,\uu_{n+1}),\\\nonumber
\widetilde G_2=&\,\tfrac{1}2\uu_n\circ\left(\uu_n+D_{2,\dx}(\beta\uu_n+\gamma D_{2,\dx}\uu_n\right),
\end{align}
where
\begin{align*}
\rho(\uu_n,\uu_{n+1})=&\,S_{\dx}^{-1}\left\{(\mu_{\dx}\mu_{\dt}\uu_n)\circ (D_{\dx}D_{\dt}\uu_n)-\tfrac{1}2D_{\dt}\big((\mu_{\dx}\uu_n)\circ(D_{\dx}\uu_n)\big)\right\},\\
\sigma(\uu_n,\uu_{n+1})=&\,\left\{(\mu_{\dt}\uu_n)\circ(D_{\dt}D_{\dx}^3S_{\dx}^{-2}\uu_n)-S_{\dx}^{-1}\big((D_{\dx}\mu_{\dt}\uu_n)\circ(D_{\dt}D_{\dx}^2\uu_n)\big)\right\}\\
&+\tfrac{1}2D_{\dt}\big\{S_{\dx}^{-1}\big((D_{\dx}\uu_n)\circ(D_{\dx}^2\uu_n)\big)-\uu_n\circ(D_{\dx}^3S_{\dx}^{-2}\uu_n)\big\}.
\end{align*}
Note that these schemes are fully implicit.  Substituting \R{eq:solver} in (\ref{MCschemes}) and differentiating in time gives,
\begin{align}\label{eq:DtMC}
\DDD_{\dt} \Phi( \dt,\uu_n,\beta,\gamma) = [\dt \mathcal{J}]^{-1}A(\mu_{\dt}\uu_n).
\end{align}
The defect and the local error estimate \R{estim} are then computed using (\ref{eq:Aop}) and (\ref{eq:DtMC}).
\begin{remark}\label{remMC}
Summation of (\ref{MCschemes}) and (\ref{MomMC}) gives that
$$\dx\sum_m u_{m,n},\qquad \dx\sum_m \left(\tfrac{1}2u_{m,n}(u_{m,n}+\beta D_{2,\dx}u_{m,n}+\gamma D_{4,\dx}u_{m,n})\right),$$
are the discretizations of the global mass and momentum, respectively. The discrete global mass is independent of the parameters, and therefore it is conserved by both  Algorithm~\ref{General}.\ref{alg:opt} and  Algorithm~\ref{General}.\ref{alg:opt2}. However, the discrete global momentum and the local conservation laws (\ref{MCschemes}) and (\ref{MomMC}) are only preserved by Algorithm~\ref{General}.\ref{alg:opt2}.
\end{remark} 
\subsection{Schemes for a nonlinear heat equation}\label{NLH}
In this section we consider the nonlinear heat equation,
\begin{equation}
\label{eq:nonlinearPDEex1}
u_t= \tfrac{1}2( u^2)_{xx},
\end{equation}
with initial condition and Dirichlet boundary conditions,
\begin{equation*}
u(x,0) = u_0(x), \quad u(a,t) = \varphi_L(t), \quad u(b,t) = \varphi_R(t).
\end{equation*}
Equation \R{eq:nonlinearPDEex1} has only two independent conservation laws, \begin{align}\label{eq:CL11}
&\DDD_xF_1+\DDD_tG_1\equiv\DDD_x(-uu_x)+\DDD_t(u)=0, \qquad \DDD_xF_2+\DDD_tG_2\equiv\DDD_x\left(\tfrac{1}2 u^2-xuu_x\right)+\DDD_t(xu)=0.
\end{align}
We denote with CS$(\lambda)$ the one-parameter family of methods described in \cite{frasca}, given by
\begin{align}\label{explNLH}
{D}_{\dt}(\uu_n+\lambda D_{2,\dx}\uu_n)=\tfrac{1}2D_{2,\dx}(\uu_n\circ\uu_{n+1}).
\end{align}
The scheme CS$(\lambda)$ has the following discrete versions of the conservation laws (\ref{eq:CL11}),
\begin{align}\label{NLHCL1}
D_{\dx}&\widetilde F_1+D_{\dt}\widetilde G_1=0,\qquad \widetilde F_1=\,-S_{\dx}^{-1}\left(\tfrac{1}2\uu_n\circ\uu_{n+1}-\lambda D_{\dt}\uu_n\right),\qquad \widetilde G_1=\uu_n,\\\label{NLHCL2}
D_{\dx}&\widetilde F_2+D_{\dt}\widetilde G_2=0,\\\nonumber
\widetilde F_2=&\,S_{\dx}^{-1}\left(\mu_{\dx}(\tfrac{1}2\uu_n\circ\uu_{n+1})-(\mu_{\dx}\xx)\circ D_{\dx}(\tfrac{1}2\uu_n\circ\uu_{n+1})\right),\qquad \widetilde G_2=\xx\circ(\uu_n+\lambda D_{2,\dx}\uu_n).
\end{align}
In order to evaluate the defect, we consider a centred difference space discretization of \R{eq:nonlinearPDEex1} in the form \R{eq:nonlinearODE} with
\begin{equation}\label{eq:Aop1}
A(\uu(t),t)=\tfrac{1}{2}\left(\hat{D}_{2,\Delta x}\uu^2(t)+\varphi_L^2(t)\MM{e}_1+\varphi_R^2(t)\MM{e}_M\right),
\end{equation}
where $\MM{e}_j$ is the $j$-th unit vector, and we define $\hat{D}_{2,\Delta x} = D_{2,\Delta x}\vert_{\varphi_L=\varphi_R=0}$ in order to isolate the contribution of the boundary conditions.
The methods in (\ref{explNLH}) are linearly implicit so they can be written in the form \R{eq:solver} with
\begin{align}
\label{eq:explicitsolverBex1}\nonumber
\Phi(\dt,\uu_n,\lambda) =&\,[\dt \mathcal{J}]^{-1}\big\{(1+\lambda \hat D_{2,\dx})\uu_n +\tfrac{\dt}{2\Delta x^2}\big(\varphi_L(t_{n+1})\varphi_L(t_n)\MM{e}_1+\varphi_R(t_{n+1})\varphi_R(t_n)\MM{e}_M\big)\\
&-\lambda\big((\varphi_L(t_{n+1})-\varphi_L(t_n))\MM{e}_1+(\varphi_R(t_{n+1})-\varphi_R(t_n))\MM{e}_M\big)\big\}.
\end{align}
Differentiating \R{eq:explicitsolverBex1} in time yields,
\begin{align*}\label{eq:DtphiB}
\nonumber \DDD_{\dt} \Phi(&\dt,\uu_n,\lambda) = [\dt \mathcal{J}]^{-1}\big\{\tfrac{1}2\hat D_{2,\dx}(\uu_n \circ \uu_{n+1})-\lambda (\varphi_L'(t_{n+1})\MM{e}_1+\varphi_R'(t_{n+1})\MM{e}_M)\\
&+\tfrac{1}{2\Delta x^2}\big((\varphi_L(t_{n+1})+\dt\varphi_L'(t_{n+1}))\varphi_L(t_n)\MM{e}_1+(\varphi_R(t_{n+1})+\dt\varphi_R'(t_{n+1}))\varphi_R(t_n)\MM{e}_M\big)\big\},
\end{align*}
which, together with \R{eq:Aop1}, is utilized in \R{estim} and \R{eq:defect} to estimate the local error.
\begin{remark}\label{remNLH}
If the boundary conditions are conservative, summing the entries of the vectors in (\ref{NLHCL1}) and (\ref{NLHCL2}) yields
$$\sum_m u_{m,n+1}=\sum_m u_{m,n},\qquad \sum_m x_m(u_{m,n+1}+\lambda D_{2,\dx}u_{m,n+1})=\sum_m x_m(u_{m,n}+\lambda D_{2,\dx}u_{m,n}),$$
giving the following approximations of the global invariants
$$\dx\sum_m u_{m,n},\qquad \dx\sum_m x_m(u_{m,n}+\lambda D_{2,\dx}u_{m,n}).$$
Since the former is independent of $\lambda$, this is a conserved invariant when using either of Algorithm~\ref{General}.\ref{alg:opt} and  Algorithm~\ref{General}.\ref{alg:opt2}. In general, the latter is conserved only by Algorithm~\ref{General}.\ref{alg:opt2}. However, if the boundary conditions are such that
$$\sum_m x_mD_{2,\dx}u_{m,n}=0,$$ the conservation of $$\dx\sum_mx_mu_{m,n}$$ is also guaranteed by Algorithm~\ref{General}.\ref{alg:opt}. On a uniform spatial grid, this is achieved, for example, with zero boundary conditions.
\end{remark}

\section{Numerical examples}\label{numerics}
In this section we consider a range of benchmark problems for the KdV equation (\ref{eq:nonlinearPDEex2}) and the nonlinear heat equation (\ref{eq:nonlinearPDEex1}), and investigate the performance of the methods described in Section~\ref{S3framework} with optimal parameters obtained by the two algorithms introduced in Section~\ref{General}. Comparisons between different numerical schemes are based on:
\begin{itemize}
\item Relative error in the solution at the final time $t=T$, defined as
$$\frac{\|\MM{u}_N-u_{\rm exact}(T)\|}{\|u_{\rm exact}(T)\|},$$
where $\|\cdot\|$ denotes the discrete $L^2$ norm and $u_{\rm exact}$ is the solution of (\ref{eq:nonlinearPDE}).
\item Error in the variation of the global densities. If the method preserves the $k$-th conservation law,
\begin{equation*}
D_{\dx} \widetilde{F}_k(\xx,\uu_n,\uu_{n+1},\chi)+D_{\dt} \widetilde{G}_k(\xx,\uu_n,\chi)=0,
\end{equation*}
the error in the global variation of $G_k$ is defined as
\begin{equation}\label{errcl}
{\rm Err}_k=\dx \max_{n=1,\ldots,N}\left| (\MM{e}_{M+1}-\MM{e}_{1})^T\widetilde{F}_k(\xx,\uu_n,\uu_{n+1},\chi)+\mathbf{1}^T D_{\dt}\widetilde{G}_k\left(\xx,\uu_n,\chi\right)\right|,
\end{equation}
where $\MM{e}_j$ denotes the $j$-th column vector of the standard basis of $\mathbb{R}^{M+1}$, and $\mathbf{1}\in \mathbb{R}^{M+1}$ is the column vector with all entries equal to one.

When the boundary conditions are periodic, we consider instead the maximum error on the $k$-th global invariant, defined as
\begin{equation}\label{errclgl}
{\rm Err}_k=\dx \max_{n=1,\ldots,N}\left| \mathbf{1}^T\left(\widetilde{G}_k\left(\xx, \uu_n,\chi\right)-\widetilde{G}_k\left(\xx, \uu_0,\chi\right)\right)\right|.
\end{equation}
\item Computational cost, measured in terms of computation time in seconds.
\end{itemize}
For each of the family of schemes described in Section~\ref{S3framework}, we consider the following choices of the vector of free parameters:
\begin{itemize}
\item $\chi=\mathbf{0}$, default choice, fixed at each step.
\item $\chi=\chi^*$, globally optimal fixed value that minimizes the solution error for this specific problem. This value is obtained by brute force search based on empirical comparisons with the exact solution and is not available a priori.
    Thus, this choice of parameters does not constitute a reasonable numerical algorithm and we do not provide the computation time for it since it is excessive.
    
\item $\{\chi_{r,n}^*\}_{n\geqslant 0}$, sequence of values that minimize the local error at each time-step obtained using Algorithm~\ref{General}.\ref{alg:opt} projecting on a grid with resolution $r\dx$.
\item $\chi=\overline{\chi}^*_r$, fixed value obtained using Algorithm~\ref{General}.\ref{alg:opt2} with projection on a coarse grid with spatial resolution $r\dx$.
\end{itemize}
As discussed in Remark~\ref{parsize}, we apply both Algorithm~\ref{General}.\ref{alg:opt} and Algorithm~\ref{General}.\ref{alg:opt2} with the simplified Gauss--Newton step (\ref{eq:iter}).

For all the experiments in this section, $\dt$ and $\dx$ are such that $4<\dt/\dx<10$. In order to verify the validity of the observations in Remark~\ref{remK}, we apply the proposed algorithms with $r=1,2,4,10$. 
\subsection{KdV equation}
In this section we solve the KdV equation (\ref{eq:nonlinearPDEex2}), comparing schemes in Section~\ref{KdV1} with different choices of the parameters against schemes known in literature -- namely, the multisymplectic method,
$$D_{\dx}\left\{\tfrac{1}2\mu_{\dx}(\mu_{\dx}\mu_{\dt}u_{m-2,n})^2+D_{2,\dx}\mu_{\dt}u_{m-1,n}\right\}+D_{\dt}\left\{\mu_{\dx}^3u_{m-2,n}\right\}=0,$$ and the narrow box scheme,
$$D_{\dx}\left\{\tfrac{1}2(\mu_{\dt}u_{m-1,n})^2+\mu_{\dt}(D_{2,\dx}u_{m-1,n})\right\}+D_{\dt}\left\{\mu_{\dx}u_{m-1,n}\right\}=0,$$
introduced in \cite{AschMCL1,AschMCL2}. Both of these schemes preserve a discrete conservation law of the mass, given by their definition, but not of the momentum or the energy. These schemes are more compact than those defined in Section~\ref{KdV1} and not centred on the grid. Therefore, we evaluate the error in the global momentum and energy according to (\ref{errclgl}) with
$$\widetilde{G}_2=\tfrac{1}2(\mu_{\dx}u_{m-1,n})^2,\qquad \widetilde{G}_3=\tfrac{1}3(\mu_{\dx}u_{m-1,n})^3+(\mu_{\dx}u_{m-1,n})D_{2,\dx}(\mu_{\dx}u_{m-1,n}).$$
For methods EC($\alpha$) (resp. MC($\beta,\gamma$)) we evaluate the error (\ref{errclgl}) in the conservation of the momentum (resp. the energy) with $\widetilde{F}_2$ and $\widetilde{G}_2$ (resp. $\widetilde{F}_3$ and $\widetilde{G}_3$) given in (\ref{MomMC}) (resp. (\ref{EnergyEC})) with all parameters set to zero.
As a first numerical test we consider the motion of a soliton. The initial condition is obtained from the exact solution on $\mathbb{R}$,
$$u(x,t)=3\,\sech^2\left(\tfrac{1}2(x-t+5)\right),$$
evaluated at $t=0$. We solve this problem over $[a,b]=[-20,20]$ till the final time $T=10$, and set a grid with $\dx=0.05$ and $\dt=0.4$.

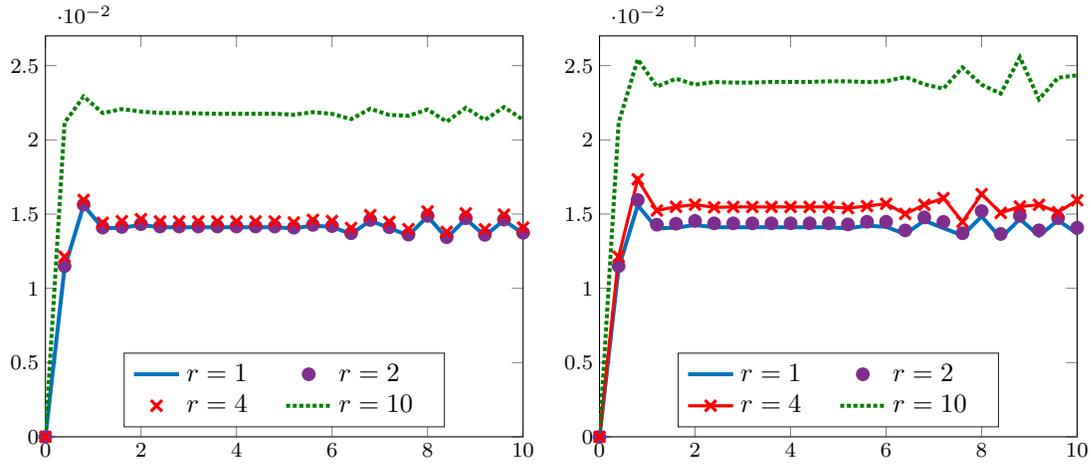
\begin{figure}[tbp]
\begin{center}
	\iftoggle{pgfplots}{\begin{tikzpicture}

\begin{axis}[%
width=2.5in,
height=2.1in,
at={(0.2in,0in)},
scale only axis,
xmin=0,
xmax=10,
ymin=0,
ymax=0.027,
axis background/.style={fill=white},
legend columns=2,
legend style={at={(0.5,0.21)},anchor=north, draw=white!15!black, text width=3.5em},
]
\addplot [color=colorclassyblue, line width = 1.5pt]
  table[row sep=crcr]{%
0	0\\
0.4	0.0113712943380846\\
0.8	0.015592447489342\\
1.2	0.0140491641672313\\
1.6	0.0140751671732834\\
2	0.0142688912067355\\
2.4	0.0141133817412694\\
2.8	0.0141133817412694\\
3.2	0.0141133817412694\\
3.6	0.0141133817412694\\
4	0.0141133817412694\\
4.4	0.0141133817412694\\
4.8	0.0141133817412694\\
5.2	0.0140388492159269\\
5.6	0.014221799136614\\
6	0.0141400260389508\\
6.4	0.0136593275033897\\
6.8	0.0145591296080188\\
7.2	0.0140573618280717\\
7.6	0.0135611191214473\\
8	0.0148411692449971\\
8.4	0.0133900857207792\\
8.8	0.0146691401907863\\
9.2	0.0135418129067595\\
9.6	0.0145812311051967\\
10	0.0136849856711994\\
};
\addlegendentry{$r=1$}

\addplot [only marks, mark=*, mark size = 2pt, mark options={solid,colorpurple},line width = 1.2pt]
  table[row sep=crcr]{%
0	0\\
0.4	0.0114976584836633\\
0.8	0.0156380433480047\\
1.2	0.0140939074771822\\
1.6	0.0141373630340395\\
2	0.0143200764598247\\
2.4	0.0141639231831168\\
2.8	0.0141639231831168\\
3.2	0.0141639231831168\\
3.6	0.0141639231831168\\
4	0.0141639231831168\\
4.4	0.0141639231831168\\
4.8	0.0141639231831168\\
5.2	0.0140906784213198\\
5.6	0.0142726799105244\\
6	0.0141920471935708\\
6.4	0.0137127193559802\\
6.8	0.0146049236047599\\
7.2	0.0141148720655181\\
7.6	0.0136180256232869\\
8	0.0148811102743438\\
8.4	0.0134450028624749\\
8.8	0.0147178323990887\\
9.2	0.0136030161122445\\
9.6	0.0146293277150773\\
10	0.0137409261885238\\
};
\addlegendentry{$r=2$}

\addplot [only marks, mark=x, mark size = 3pt, mark options={solid,colorabs},  line width = 1.2pt]
  table[row sep=crcr]{%
0	0\\
0.4	0.0121299597112036\\
0.8	0.0159548465177579\\
1.2	0.0144155739178454\\
1.6	0.0145229566771103\\
2	0.0146604583488011\\
2.4	0.0145039098875163\\
2.8	0.0145039098875163\\
3.2	0.0145039098875163\\
3.6	0.0145039098875163\\
4	0.0145039098875163\\
4.4	0.0145039098875163\\
4.8	0.0145039098875163\\
5.2	0.0144338821450387\\
5.6	0.0146118483717539\\
6	0.0145347450567334\\
6.4	0.014062103014043\\
6.8	0.0149264712208969\\
7.2	0.0144738819266635\\
7.6	0.0139875000902516\\
8	0.0151755698736747\\
8.4	0.0137996990836454\\
8.8	0.0150504763036823\\
9.2	0.0139758025389692\\
9.6	0.0149642601520507\\
10	0.0140930963155762\\
};
\addlegendentry{$r=4$}

\addplot [color=colorimag, densely dotted, line width = 1.5pt]
  table[row sep=crcr]{%
0	0\\
0.4	0.0211488924016794\\
0.8	0.022924717554602\\
1.2	0.0218035841470419\\
1.6	0.0220736520940138\\
2	0.0219047435405373\\
2.4	0.0218129476338356\\
2.8	0.0218129476338356\\
3.2	0.021785301826166\\
3.6	0.0217570754611966\\
4	0.0217570754611966\\
4.4	0.0217570754611966\\
4.8	0.0217570754611966\\
5.2	0.0216964147700414\\
5.6	0.0218653918357601\\
6	0.0217523652320602\\
6.4	0.0213935286451561\\
6.8	0.0220950865856584\\
7.2	0.0216775867710507\\
7.6	0.021620255250616\\
8	0.022044392295746\\
8.4	0.0212225524165475\\
8.8	0.0221529700274654\\
9.2	0.0213408853309564\\
9.6	0.022207930889596\\
10	0.021354688020783\\
};
\addlegendentry{$r=10$}

\end{axis}

\begin{axis}[%
width=2.5in,
height=2.1in,
at={(3.1in,0in)},
scale only axis,
xmin=0,
xmax=10,
ymin=0,
ymax=0.027,
axis background/.style={fill=white},
legend columns=2,
legend style={at={(0.5,0.21)},anchor=north, draw=white!15!black,text width=3.5em},
]
\addplot [color=colorclassyblue,  line width = 1.5pt]
  table[row sep=crcr]{%
0	0\\
0.4	0.0113712943380846\\
0.8	0.015592447489342\\
1.2	0.0140491641672313\\
1.6	0.0140751671732834\\
2	0.0142688912067355\\
2.4	0.0141133817412694\\
2.8	0.0141133817412694\\
3.2	0.0141133817412694\\
3.6	0.0141133817412694\\
4	0.0141133817412694\\
4.4	0.0141133817412694\\
4.8	0.0141133817412694\\
5.2	0.0140388492159269\\
5.6	0.014221799136614\\
6	0.0141400260389508\\
6.4	0.0136593275033897\\
6.8	0.0145591296080188\\
7.2	0.0140573618280717\\
7.6	0.0135611191214473\\
8	0.0148411692449971\\
8.4	0.0133900857207792\\
8.8	0.0146691401907863\\
9.2	0.0135418129067595\\
9.6	0.0145812311051967\\
10	0.0136849856711994\\
};
\addlegendentry{$r=1$}

\addplot [only marks, mark=*, mark size = 2pt, mark options={solid,colorpurple},line width = 1.2pt]
  table[row sep=crcr]{%
0	0\\
0.4	0.0114976584836633\\
0.8	0.0159471300450079\\
1.2	0.0142858050368536\\
1.6	0.0143510370723648\\
2	0.0145414729401956\\
2.4	0.0143778184087228\\
2.8	0.0143778184087228\\
3.2	0.0143778184087228\\
3.6	0.0143778184087228\\
4	0.0143778184087228\\
4.4	0.0143778184087228\\
4.8	0.0143778184087228\\
5.2	0.0143053245590477\\
5.6	0.0144819000320423\\
6	0.0144819000320423\\
6.4	0.0139020843908288\\
6.8	0.0147768704550733\\
7.2	0.0144753099648095\\
7.6	0.013712410530806\\
8	0.0152002356836506\\
8.4	0.013663275220558\\
8.8	0.0148860642458386\\
9.2	0.0139077714693748\\
9.6	0.014728947098847\\
10	0.0140756679002216\\
};
\addlegendentry{$r=2$}

\addplot [colorabs, mark=x, mark size = 3pt, mark options={solid,colorabs},  line width = 1.2pt]
  table[row sep=crcr]{%
0	0\\
0.4	0.0121299597112036\\
0.8	0.0173413636271164\\
1.2	0.0152427029825241\\
1.6	0.0154928232610698\\
2	0.0156370935117993\\
2.4	0.0154540177913506\\
2.8	0.0154805484178851\\
3.2	0.0154805484178851\\
3.6	0.0154805484178851\\
4	0.0154805484178851\\
4.4	0.0154805484178851\\
4.8	0.0154805484178851\\
5.2	0.0154054505486054\\
5.6	0.0155368688878111\\
6	0.0156960378202518\\
6.4	0.015014859828873\\
6.8	0.0156233422869195\\
7.2	0.0160747748199808\\
7.6	0.0144789251200603\\
8	0.0163448811576515\\
8.4	0.0150800387702957\\
8.8	0.0155095605430619\\
9.2	0.0156256041819161\\
9.6	0.0151278272533317\\
10	0.0159325061358721\\
};
\addlegendentry{$r=4$}

\addplot [color=colorimag, densely dotted, line width = 1.5pt]
  table[row sep=crcr]{%
0	0\\
0.4	0.0211488924016794\\
0.8	0.0254282750258451\\
1.2	0.0235865657493124\\
1.6	0.0241204248486895\\
2	0.0237192897985163\\
2.4	0.0238977060110098\\
2.8	0.0238585780569699\\
3.2	0.0238585780569699\\
3.6	0.0239062235318405\\
4	0.0239062235318405\\
4.4	0.0239062235318405\\
4.8	0.0239446019217911\\
5.2	0.0239446019217911\\
5.6	0.0238902912941283\\
6	0.0239485134292638\\
6.4	0.0242390653305657\\
6.8	0.0237438688890532\\
7.2	0.023449936569339\\
7.6	0.0248795282223649\\
8	0.0237151996633765\\
8.4	0.0231144651688539\\
8.8	0.0255618674742629\\
9.2	0.022732494726632\\
9.6	0.0241835163544818\\
10	0.0243552495922895\\
};
\addlegendentry{$r=10$}

\end{axis}

\end{tikzpicture}%
}{}
\end{center}
	\caption{One-soliton problem for KdV (\ref{eq:nonlinearPDEex2}). Sequence of parameters $\{\alpha_{r,n}^*\}_{n\geqslant 0}$ for EC($\alpha$), (\ref{ECschemes}), obtained using Algorithm~\ref{General}.\ref{alg:opt} (left) and Algorithm~\ref{General}.\ref{alg:opt2} (right) with different values of $r$.}
	\label{a1}
\end{figure}
\begin{figure}[tbp]
\begin{center}
	\iftoggle{pgfplots}{\input{bg1.tex}}{}
\end{center}
	\caption{One-soliton problem for KdV (\ref{eq:nonlinearPDEex2}). Sequence of parameters $\{(\beta_{r,n}^*,\gamma_{r,n}^*)\}_{n\geqslant 0}$ for MC($\beta,\gamma$), (\ref{MCschemes}), obtained using Algorithm~\ref{General}.\ref{alg:opt} (left) and Algorithm~\ref{General}.\ref{alg:opt2} (right) with different values of $r$.}
	\label{bg1}
\end{figure}
\begin{table}[t]
\caption{One-soliton problem for KdV (\ref{eq:nonlinearPDEex2}). Errors in solution and conservation laws, and computation time.}\label{tab:errtabKdV}
\small
\begingroup
\setlength{\tabcolsep}{6pt}
\renewcommand{\arraystretch}{1.12} 
\centerline{\begin{tabular}{|c|c|c|c|c|c|}
\hline
Method &  $\text{Err}_1$ & $\text{Err}_2$ & $\text{Err}_3$ & Sol. err. & Comput. time\\ 
\hline
\hline
EC$(\alpha_{1,n}^*)_{n\geq 0}$	&  1.88e-13   & 3.24e-04 & 6.25e-13 & 0.0139  & 32.33\\	
\hline
EC$(\alpha_{2,n}^*)_{n\geq 0}$	&  1.92e-13   & 3.27e-04 & 6.43e-13 & 0.0139  & 14.43\\	
\hline
EC$(\alpha_{4,n}^*)_{n\geq 0}$	&  1.31e-13   & 3.50e-04 & 4.51e-13 & 0.0133  & 8.06\\	
\hline
EC$(\alpha_{10,n}^*)_{n\geq 0}$	&  2.03e-13   & 0.0010 & 1.06e-12 & 0.0092  & 6.51\\	
\hline
\hline
EC$(\overline{\alpha}^*_1)\equiv$ EC$(0.014)$	&  3.11e-13   & 3.28e-04 & 1.16e-12 & 0.0140 & 42.12\\	
\hline
EC$(\overline{\alpha}^*_2)\equiv$ EC$(0.014)$	&  2.84e-13   & 3.44e-04 & 8.14e-12 & 0.0136 & 16.39\\	
\hline
EC$(\overline{\alpha}^*_4)\equiv$ EC$(0.015)$	&  1.97e-13   & 4.15e-04 & 6.91e-13 & 0.0120 & 9.65\\
\hline
EC$(\overline{\alpha}^*_{10})\equiv$ EC$(0.024)$	&  1.10e-13   & 0.0013 & 5.81e-13 & 0.0113 & 7.10\\		
\hline
\hline
MC$(\beta_{1,n}^*,\gamma_{1,n}^*)_{n\geq 0}$	&  9.33e-13   & 0.0017 & 4.90e-04 & 0.0135  & 37.38\\	
\hline
MC$(\beta_{2,n}^*,\gamma_{2,n}^*)_{n\geq 0}$	&  6.64e-13   & 0.0017 & 5.08e-04 & 0.0140  & 16.26\\	
\hline
MC$(\beta_{4,n}^*,\gamma_{4,n}^*)_{n\geq 0}$	&  3.32e-13   & 0.0019 & 5.70e-04 & 0.0161  & 8.00\\	
\hline
MC$(\beta_{10,n}^*,\gamma_{10,n}^*)_{n\geq 0}$	&  1.92e-13   & 0.0028 & 9.45e-04 & 0.0242  & 6.95\\	
\hline
\hline
MC$(\overline{\beta}^*_1,\overline{\gamma}^*_{1})\equiv$ MC$(0.045,0.016)$	&  5.61e-13   & 3.03e-12 & 5.01e-04 & 0.0135 & 50.07\\	
\hline
MC$(\overline{\beta}^*_2,\overline{\gamma}^*_{2})\equiv$ MC$(0.046,0.017)$	&  5.12e-13   & 3.33e-12 & 5.08e-04 & 0.0135 & 17.56\\	
\hline
MC$(\overline{\beta}^*_4,\overline{\gamma}^*_{4})\equiv$ MC$(0.048,0.018)$	&  4.90e-13   & 4.35e-12 & 5.36e-04 & 0.0138 & 9.11\\	
\hline
MC$(\overline{\beta}^*_{10},\overline{\gamma}^*_{10})\equiv$ MC$(0.063,0.030)$	&  7.14e-13   & 4.44e-12 & 8.00e-04 & 0.0181 & 7.74\\	
\hline
\hline
EC$(0)$	& 1.39e-13   & 1.54e-04 & 3.27e-13 &  0.0376 & 6.07 \\	
\hline
MC$(0,0)$	& 2.34e-13   & 4.67e-13 & 3.04e-04 &  0.0446 & 6.07 \\	
\hline
EC$(\alpha^*)\equiv$ EC$(0.020)$	&  1.39e-13   & 8.26e-04 & 4.14e-13 & 0.0085 & N/A \\
\hline
MC$(\beta^*,\gamma^*)\equiv$ MC$(0.055,0.031)$	&  1.23e-12   & 7.27e-12 & 7.57e-04 & 0.0083 & N/A \\
\hline
Multisymplectic&  9.52e-13   & 6.40e-06 & 2.73e-04 & 0.0447 & 8.55\\		
\hline
Narrow box&  2.22e-12   & 2.39e-06 & 2.90e-04 & 0.0434 & 7.24\\
\hline
\end{tabular}}
\endgroup
\end{table}
We first find estimates for the optimal parameters of the schemes EC($\alpha$) and MC($\beta,\gamma$) by applying the two new algorithms. In Figure~\ref{a1} we show the sequences of optimal values given by Algorithm~\ref{General}.\ref{alg:opt} and Algorithm~\ref{General}.\ref{alg:opt2}, respectively, for the scheme EC. Similarly, in Figure~\ref{bg1} we show the sequences obtained for the scheme MC.

Since the solution only travels along the same direction with constant speed, after a few initial steps the sequences of parameters stabilize around constant values.

The sequences obtained by the two algorithms with $r=1,2,4,$ are all very close to each other (within a maximum distance of $5\cdot 10^{-3}$). In agreement with Remark~\ref{remK}, those obtained with $r=10>\dt/\dx$ show the effect of the space error on the coarser grid. This shows that the accuracy in the solution can be compromised when $r$ is too large. Nevertheless, these parameters are not too far from the values obtained on the computational grid, and still show a good level of accuracy in all cases, further reducing the computational overheads involved in the solution of the optimization problem.

In Table~\ref{tab:errtabKdV} we compare different schemes. The results obtained show that:
\begin{itemize}
\item[•] Schemes EC and MC with fixed values of the parameters exactly preserve two conservation laws, and therefore two global invariants.
\item[•] According to Remark~\ref{remEC}, due to the conservative boundary conditions, the sequence of schemes EC$(\alpha^*_{r,n})_{n\geqslant 0}$ preserves the global mass and the global energy. Similarly, according to Remark~\ref{remMC}, schemes MC$(\beta^*_{r,n},\gamma^*_{r,n})_{n\geqslant 0}$ preserve the global mass but not the global momentum.
\item[•] The computational cost of both of the proposed algorithms decreases as $r$ increases. With respect to solving the optimization problem on the same grid of the differential problem ($r=1$), the overall computational time reduces to less than a half when $r=2$, and to less than a fourth when $r=4$, making the computation cost of the methods proposed in this paper comparable to that of the other schemes in literature. As expected, setting $r=10$, the computational time further reduces.
\item[•] All the approximations obtained with the sequences of parameters given by Algorithm~\ref{General}.\ref{alg:opt} and Algorithm~\ref{General}.\ref{alg:opt2} are more accurate than the solutions of the schemes from the literature and of EC(0) and MC(0,0). However, schemes EC($\alpha^*$) and MC($\beta^*,\gamma^*$), where the parameters are obtained by brute force, are more precise. This shows that there exist sequences of parameters yielding higher accuracy than those obtained using the new algorithms. This is due to the fact that our approach is based on a minimization of (\ref{estim}) as an estimate of the local error. This improves local error but does not necessarily lead to a minimization of the global error.
\item[•] The accuracy of the solutions obtained using the two new algorithms is only marginally affected by the different choices of $r<\dt/\dx$. However, for $r=10>\dt/\dx$ the spatial error has a visible effect on the sequences of parameters obtained, as is evident in Figures~\ref{a1} and~\ref{bg1}. In this case, the accuracy in the solution is only marginally affected but in general this may not be the case. In agreement with Remark~\ref{remK}, setting $r=4$ gives the best compromise between reliability and speed of computation.

\item[•] Algorithm~\ref{General}.\ref{alg:opt} and Algorithm~\ref{General}.\ref{alg:opt2} with $r=4$ are significantly more efficient than the schemes from the literature: while computation times are comparable, the solution error is roughly three times lower.
\end{itemize}
\begin{figure}[tbp]
\begin{center}
	\iftoggle{pgfplots}{
	\input{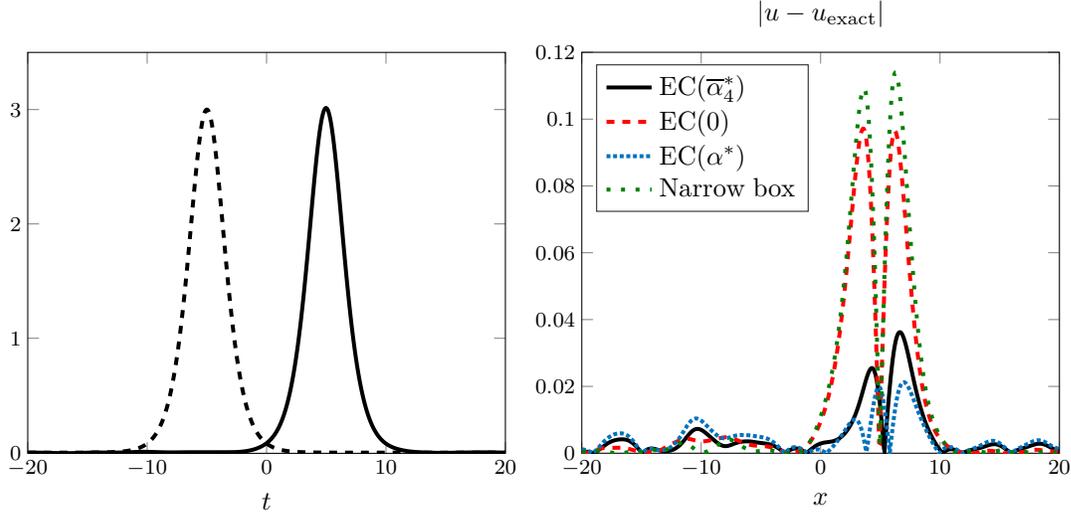} }{}
\end{center}
	\caption{One-soliton problem for KdV (\ref{eq:nonlinearPDEex2}). Left: Initial condition (dashed line) and solution of EC($\overline\alpha^*_{4})$. Right: Pointwise error for different schemes at the final time.}
	\label{figkdv1}
\end{figure}
In the left plot of Figure~\ref{figkdv1} we show the initial profile and, as an example, the solution of EC($\overline{\alpha}_{4}^*$) at the final time. In the right plot we show the absolute error given by EC($\overline{\alpha}_{4}^*$) at every point, in comparison with EC(0), EC($\alpha^*$), and the narrow box scheme. For all schemes, the bulk of the error is detected around the final location of the soliton and is due to a delay introduced by all numerical schemes. However, the maximum error introduced by EC($\overline{\alpha}_{4}^*$) is less than the 40\% of that given by EC(0) and by the narrow box scheme, and only slightly larger than the error given by EC($\alpha^*$).

As a second numerical test we consider the interaction of two solitons over $[a,b]=[-30,30]$ and till time $T=15$. The initial condition is obtained from the exact solution on $\mathbb{R}$,
$$u(x,t)=\frac{12(c_1-c_2)(c_1\cosh^2\xi_2+c_2\sinh^2\xi_1)}{\left((\sqrt{c_1}-\sqrt{c_2})\cosh(\xi_1+\xi_2)+(\sqrt{c_1}+\sqrt{c_2})\cosh(\xi_1-\xi_2)\right)^2},$$
with $$\xi_1=\frac{c_1}2(x+d_1-c_1t),\quad\xi_2=\frac{c_2}2(x+d_2-c_2t).$$
We set
$$c_1=2,\qquad c_2=1,\qquad d_1=17,\qquad d_2=10,\qquad \dx=0.05,\qquad \dt=0.25.$$

\begin{figure}[tbp]
\begin{center}
	\iftoggle{pgfplots}{\input{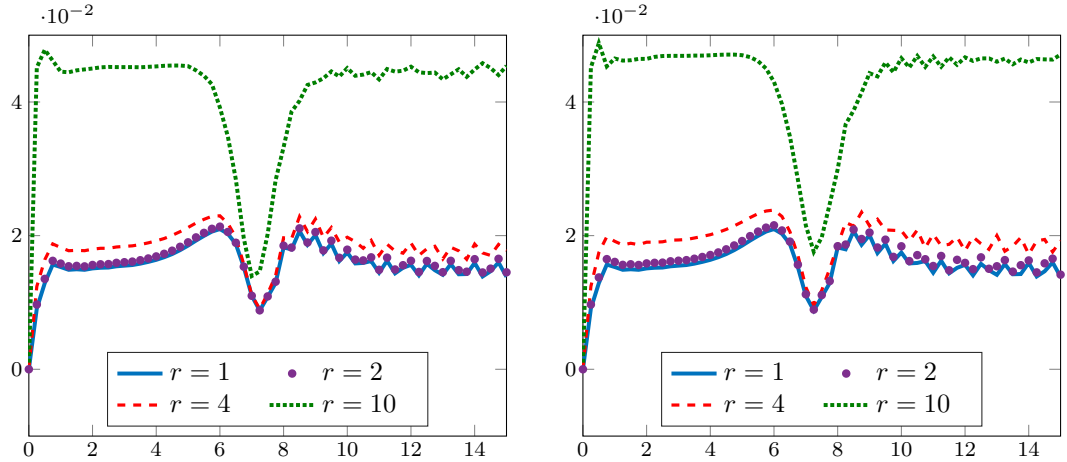}
	}{}
\end{center}
	\caption{Two-soliton problem for KdV (\ref{eq:nonlinearPDEex2}). Sequence of parameters $\{\alpha_{r,n}^*\}_{n\geqslant 0}$ for EC($\alpha$), (\ref{ECschemes}), obtained using Algorithm~\ref{General}.\ref{alg:opt} (left) and Algorithm~\ref{General}.\ref{alg:opt2} (right) with different values of $r$.}
	\label{a2}
\end{figure}
\begin{figure}[tbp]
\begin{center}
	\iftoggle{pgfplots}{\input{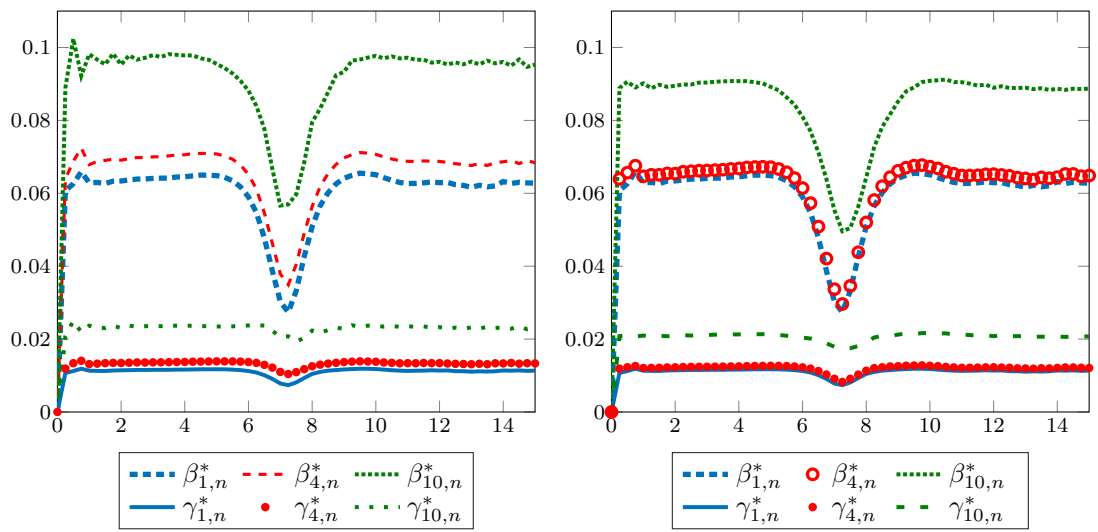}}{}
\end{center}
	\caption{Two-soliton problem for KdV (\ref{eq:nonlinearPDEex2}). Sequence of parameters $\{(\beta_{r,n}^*,\gamma_{r,n}^*)\}_{n\geqslant 0}$ for MC($\beta,\gamma$), (\ref{MCschemes}), obtained using Algorithm~\ref{General}.\ref{alg:opt} (left) and Algorithm~\ref{General}.\ref{alg:opt2} (right) with different values of $r$.}
	\label{bg2}
\end{figure}
In Figures~\ref{a2} and \ref{bg2} we show the sequences of parameters obtained for the schemes EC and MC, respectively. For sake of clarity, we omit in Figure~\ref{bg2} the sequences obtained for $r=2$ as these would be indistinguishable from those obtained setting $r=1$. Again, all the sequences obtained by solving the optimization problems on a grid up to four times coarser are very close (within a distance of $6\cdot 10^{-3}$) to those obtained on the computational grid.

We notice that the parameters rapidly change when the solitons interact, but only slowly variate before and after the interaction.
\begin{table}[t]
\caption{Two-soliton problem for KdV (\ref{eq:nonlinearPDEex2}). Errors in solution and conservation laws, and computation time.}\label{tab:errtabKdV2}
\small
\begingroup
\setlength{\tabcolsep}{6pt} 
\renewcommand{\arraystretch}{1.12} 
\centerline{\begin{tabular}{|c|c|c|c|c|c|}
\hline
Method &  $\text{Err}_1$ & $\text{Err}_2$ & $\text{Err}_3$ & Sol. err. & Comput. time\\ 
\hline
\hline
EC$(\alpha_{1,n}^*)_{n\geq 0}$	&   2.81e-13  & 0.0490 & 4.79e-12 & 0.1898 & 364.25\\	
\hline
EC$(\alpha_{2,n}^*)_{n\geq 0}$	&  2.10e-13   & 0.0453 &  5.91e-12 & 0.1856  & 127.11 \\	
\hline
EC$(\alpha_{4,n}^*)_{n\geq 0}$	& 6.57e-13    & 0.0287 & 9.82e-12 & 0.1664  & 85.69\\	
\hline
EC$(\alpha_{10,n}^*)_{n\geq 0}$	& 3.30e-13    & 0.1439  & 4.04e-12 & 0.1275  & 84.41\\	
\hline
\hline
EC$(\overline\alpha^*_{1})\equiv$ EC$(0.016)$	&  5.90e-13   & 0.0697 & 5.85e-12 & 0.1897  & 474.40 \\	
\hline
EC$(\overline\alpha^*_{2})\equiv$ EC$(0.017)$	&  1.99e-13   & 0.0637  & 2.56e-12  &  0.1841 & 166.42\\	
\hline
EC$(\overline\alpha^*_{4})\equiv$ EC$(0.019)$	& 5.12e-13    & 0.0384 & 6.12e-12 & 0.1604 & 97.08\\
\hline
EC$(\overline\alpha^*_{10})\equiv$ EC$(0.043)$	& 8.95e-13    & 0.1768 & 1.45e-11 & 0.1250 & 93.81\\		
\hline
\hline
MC$(\beta_{1,n}^*,\gamma_{1,n}^*)_{n\geq 0}$	&  1.21e-12   & 0.4128  & 0.9188  & 0.0834  & 416.18 \\	
\hline
MC$(\beta_{2,n}^*,\gamma_{2,n}^*)_{n\geq 0}$	&  2.06e-12   & 0.4187  & 0.9183  & 0.0825  & 175.44 \\	
\hline
MC$(\beta_{4,n}^*,\gamma_{4,n}^*)_{n\geq 0}$	& 2.49e-12    & 0.4404 & 0.9165 & 0.0812  & 94.32\\	
\hline
MC$(\beta_{10,n}^*,\gamma_{10,n}^*)_{n\geq 0}$	&  3.24e-12   & 0.5410  & 0.9082 & 0.1252  & 86.30\\	
\hline
\hline
MC$(\overline\beta^*_{1},\overline\gamma^*_{1})\equiv$ MC$(0.0602,0.0111)$	&  1.49e-12   & 1.66e-12  & 0.9269  & 0.0874  & 480.98 \\	
\hline
MC$(\overline\beta^*_{2},\overline\gamma^*_{2})\equiv$ MC$(0.0606,0.0113)$	& 2.38e-12    & 1.85e-11  & 0.9266 & 0.0873 & 177.06\\	
\hline
MC$(\overline\beta^*_{4},\overline\gamma^*_{4})\equiv$ MC$(0.0624,0.0120)$	&  1.15e-12   & 1.48e-11 & 0.9256 & 0.0856 & 92.42 \\	
\hline
MC$(\overline\beta^*_{10},\overline\gamma^*_{10})\equiv$ MC$(0.0848,0.0207)$	& 4.22e-12   & 1.39e-11 & 0.9198 & 0.0918 & 86.05 \\	
\hline
\hline
EC$(0)$	&  4.44e-13   & 0.2186 & 5.44e-12 & 0.3208  & 78.11 \\	
\hline
MC$(0,0)$	& 2.03e-13   & 6.04e-13 & 0.8567 &  0.3884 & 74.59 \\	
\hline
EC$(\alpha^*)\equiv$ EC$(0.034)$	& 7.07e-13    & 0.0951 & 1.01e-11 & 0.0683 & N/A \\
\hline
MC$(\beta^*,\gamma^*)\equiv$ MC$(0.147,0.065)$	& 7.90e-12    & 7.10e-11 & 0.8318 & 0.0689 & N/A \\
\hline
Multisymplectic&  2.91e-12   & 0.0081 & 0.8807 & 0.3885 & 94.28\\		
\hline
Narrow box&  4.50e-12   & 0.0041 & 0.8160 & 0.3825 & 82.44\\
\hline
\end{tabular}}
\endgroup
\end{table}

In Table~\ref{tab:errtabKdV2} we compare the accuracy, efficiency and conservation properties of different schemes. The same observations done for the case of the motion of a solitary wave hold also in this case. As before, the computational times of the new algorithms with $r=4$ are comparable to those of the methods from the literature, and their greater efficiency is evident through solution errors that are much lower.
\begin{figure}[tbp]
\begin{center}
	\iftoggle{pgfplots}{\input{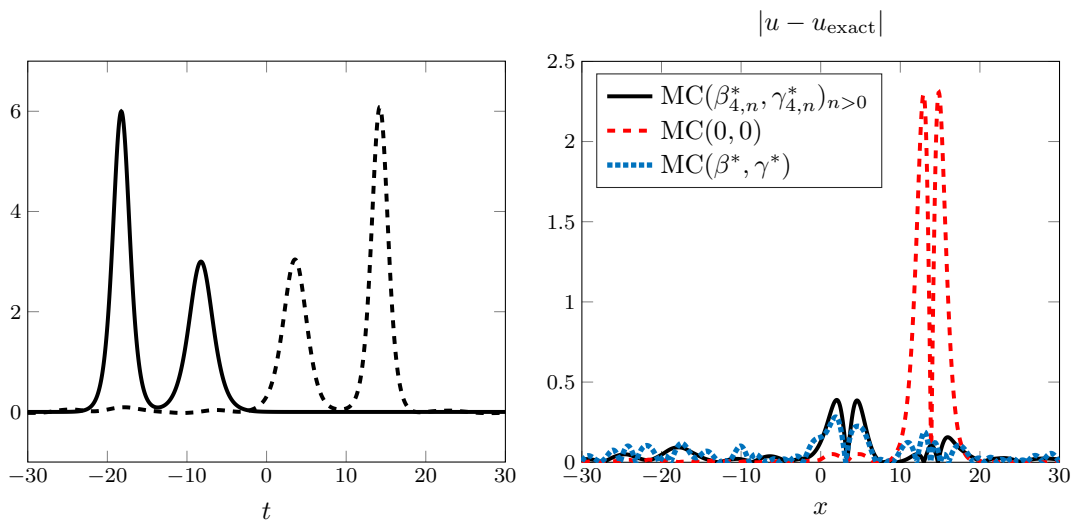}}{}
\end{center}
	\caption{Two-soliton problem for KdV (\ref{eq:nonlinearPDEex2}). Left: Initial condition (dashed line) and solution of MC$(\beta_{4,n}^*,\gamma_{4,n}^*)_{n\geq 0}$. Right: Pointwise error for different schemes at the final time.}
	\label{figkdv2}
\end{figure}

In the left plot of Figure~\ref{figkdv2}, we show the initial condition together with the solution of the sequence MC$(\beta_{4,n}^*,\gamma_{4,n}^*)_{n\geq 0}$. In the right plot, we show the pointwise error of MC$(\beta_{4,n}^*,\gamma_{4,n}^*)_{n\geq 0}$, MC(0,0), and MC$(\beta^*,\gamma^*)$ at the final time. We omit the results for the multisymplectic and the narrow box scheme, as they are similar to that of MC(0,0). The methods obtained using the approaches introduced in this paper are very accurate around the final location of the faster soliton, where the bulk of the error given by MC$(0,0)$ and by the schemes from the literature is located. However, the error around the slower soliton is larger and small oscillations can be seen far from the solitons. MC$(\beta^*,\gamma^*)$ gives a slightly smaller error around the slower soliton, but more oscillations are introduced. 
\subsection{Nonlinear heat equation}
In this section we apply the new algorithms to the nonlinear heat equation (\ref{eq:nonlinearPDEex1}) using the methods CS($\lambda$) from Section~\ref{NLH}.

We consider here two benchmark problems with (weak) energy solutions that are not classic solutions, having at least one point of non differentiability. Although in Section~\ref{General} smoothness of the solution is assumed, we show that the strategies introduced in this paper are also effective in this setting.

In order to converge, explicit and implicit finite difference methods for (\ref{eq:nonlinearPDEex1}) in literature typically require $\dt=\mathcal{O}(\dx^2)$ and $\dt=\mathcal{O}(\dx)$, respectively \cite{delTeso,delTeso2,grav,gurtin,hoff}. Under such small time-step restrictions, the Crank-Nicolson method applied to the semidiscretization (\ref{eq:Aop1}) turns out to be very accurate and efficient. However, it fails to converge for the two benchmark tests in this section with $\dt>\dx$. Such instabilities may also occur when using a CS($\lambda$) method with a default choice of the parameter, $\lambda=0$. In contrast, we find that the two proposed algorithms in Section~\ref{General}, based on optimization of defect based error estimate, are able to avoid these instabilities.

The first benchmark problem is given by equation \R{eq:nonlinearPDEex1} with initial and boundary conditions,
\begin{equation}\label{eq:datatest}
u(x,0)=0,\qquad u(0,t)=t,\qquad u(5,t)=0,\qquad (x,t)\in [0,6]\times[0,3].
\end{equation}
The solution of this problem is a linear wave travelling in an undisturbed medium with unit speed,
\begin{equation*}
u_{\text{exact}}(x,t)=\left\{\begin{array}{ll}
t-x,& \text{if } t>x,\\
0,& \text{otherwise.}
\end{array}\right.
\end{equation*}
We discretize the initial boundary value problem described by \R{eq:nonlinearPDEex1} and \R{eq:datatest}, choosing $\dx=0.025$ and $\dt=0.12$. The graphs in Figure~\ref{fig:advs0} show the sequences of parameters obtained from Algorithm~\ref{General}.\ref{alg:opt} and Algorithm~\ref{General}.\ref{alg:opt2}. Although the values are different for the first time-steps, for the smaller values of $r$ the two procedures converge to the same value.

In Table~\ref{tab:errtab} we compare the different choices for the parameter $\lambda$. The error in the conservation laws is calculated according to (\ref{errcl}) and the figures in the table show that these are preserved to machine accuracy by all the methods that use a fixed value of the free parameter.

\begin{table}[t]
\caption{NLH (\ref{eq:nonlinearPDEex1}) with initial and boundary conditions (\ref{eq:datatest}). Errors in solution and conservation laws and computation time.}\label{tab:errtab}
\small
\centerline{\begin{tabular}{|c|c|c|c|c|}
\hline
$\lambda$ &  $\text{Err}_1$ & $\text{Err}_2$ & Solution error & Computation time\\ 
\hline
\hline
$\{\lambda_{1,n}^*\}_{n\geqslant 0}$	&  0.1258   & 0.0190 & 0.0054 & 0.136 \\	
\hline
$\{\lambda_{2,n}^*\}_{n\geqslant 0}$	&  0.1182   & 0.0169 & 0.0052 & 0.069 \\
\hline
$\{\lambda_{4,n}^*\}_{n\geqslant 0}$	&  0.1030   & 0.0132 & 0.0047 & 0.055 \\
\hline
$\{\lambda_{10,n}^*\}_{n\geqslant 0}$	&  0.0817   & 0.0127 & 0.0027 & 0.035 \\
\hline
\hline
$\overline\lambda_{1}^*= -0.0115$	&  1.18e-13   & 2.98e-14 & 0.0054 & 0.122 \\
\hline
$\overline\lambda_{2}^*=-0.0110$	&  8.06e-14   & 3.20e-14 & 0.0052 & 0.077 \\	
\hline
$\overline\lambda_{4}^*=-0.0096$	&  1.15e-13   & 4.48e-14 & 0.0046 & 0.046 \\
\hline
$\overline\lambda_{10}^*= -0.0064$ &  6.68e-14   & 5.12e-14 & 0.0031 & 0.028 \\
\hline
\hline
$\lambda=0$	& 6.54e-14   & 9.24e-15 & 7.5017 &  0.016 \\
\hline
$\lambda^*=-0.0044$	&  9.24e-14   & 3.27e-14 & 0.0023 & N/A\\
\hline
\end{tabular}}
\end{table}
\begin{figure}[tbp]
\begin{center}
	\iftoggle{pgfplots}{\begin{tikzpicture}

\begin{axis}[%
width=2.5in,
height=2.1in,
at={(0.2in,0in)},
scale only axis,
xmin=0,
xmax=3,
ymin=-0.02,
ymax=0.0002,
axis background/.style={fill=white},
legend columns=2,
legend style={at={(0.5,0.21)},anchor=north, draw=white!15!black, text width=3.5em},
]
\addplot [color=colorclassyblue, line width = 1.5pt]
  table[row sep=crcr]{%
0	0\\
0.12	-0.0189738466667441\\
0.24	-0.0189738466667441\\
0.36	-0.0138703086277175\\
0.48	-0.0123117063829757\\
0.6	-0.0112982235722997\\
0.72	-0.0107978084408915\\
0.84	-0.0106692747340432\\
0.96	-0.0107297162091317\\
1.08	-0.0108702291007542\\
1.2	-0.0110184216371119\\
1.32	-0.0111411053928351\\
1.44	-0.0112266748289858\\
1.56	-0.0112786443936955\\
1.68	-0.0113047625484627\\
1.8	-0.0113163318028231\\
1.92	-0.0113189773182982\\
2.04	-0.0113190854782163\\
2.16	-0.0113182229047709\\
2.28	-0.01131764408863\\
2.4	-0.0113177594042297\\
2.52	-0.0113185308482294\\
2.64	-0.011319725482942\\
2.76	-0.0113211052569178\\
2.88	-0.0113224792407227\\
3	-0.0113237352976592\\
};
\addlegendentry{$r=1$}

\addplot [only marks, mark=*, mark size = 2pt, mark options={solid,colorpurple},line width = 1.2pt]
  table[row sep=crcr]{%
0	0\\
0.12	-0.0169368088340224\\
0.24	-0.0169368088340224\\
0.36	-0.012877445296241\\
0.48	-0.0118562943127849\\
0.6	-0.0110864550793771\\
0.72	-0.0106903485255723\\
0.84	-0.0105738239242116\\
0.96	-0.0106044252938144\\
1.08	-0.0107003754247079\\
1.2	-0.0108050227431159\\
1.32	-0.0108931768661039\\
1.44	-0.0109557403107\\
1.56	-0.0109935345784145\\
1.68	-0.0110146133451515\\
1.8	-0.0110240829859255\\
1.92	-0.0110267080417449\\
2.04	-0.0110273005233522\\
2.16	-0.0110269594254419\\
2.28	-0.0110266531545275\\
2.4	-0.0110267501452676\\
2.52	-0.0110272723975844\\
2.64	-0.0110280880426057\\
2.76	-0.0110290430315337\\
2.88	-0.0110300084507137\\
3	-0.0110309032951023\\
};
\addlegendentry{$r=2$}

\addplot [colorabs, mark=x, mark size = 3pt, mark options={solid,colorabs},  line width = 1.2pt]
  table[row sep=crcr]{%
0	0\\
0.12	-0.0131839473189761\\
0.24	-0.0131839473189761\\
0.36	-0.0106724157036676\\
0.48	-0.0103374257818605\\
0.6	-0.00994255416697802\\
0.72	-0.00974756648675125\\
0.84	-0.00972723859910398\\
0.96	-0.00974061262135248\\
1.08	-0.00972191291680772\\
1.2	-0.0097127279880946\\
1.32	-0.00976930001795233\\
1.44	-0.00985884102076011\\
1.56	-0.00989897981644514\\
1.68	-0.00986560427899413\\
1.8	-0.00981879605933818\\
1.92	-0.00983367227898907\\
2.04	-0.0098952613124933\\
2.16	-0.00992039413849568\\
2.28	-0.00987836318458544\\
2.4	-0.00982460419248923\\
2.52	-0.0098366547481595\\
2.64	-0.00989698730685796\\
2.76	-0.00992244108442744\\
2.88	-0.00988089323770416\\
3	-0.00982705565897342\\
};
\addlegendentry{$r=4$}

\addplot [color=colorimag, densely dotted, line width = 1.5pt]
  table[row sep=crcr]{%
0	0\\
0.12	-0.00871330246895771\\
0.24	-0.0126778048559879\\
0.36	-0.000335666016537431\\
0.48	-0.00571422378488531\\
0.6	-0.00268505509234015\\
0.72	-0.00545233420291492\\
0.84	-0.004745817444437\\
0.96	-0.00543598811639193\\
1.08	-0.00529911079208862\\
1.2	-0.00506874134406545\\
1.32	-0.00597485321460424\\
1.44	-0.00468679867397076\\
1.56	-0.00652733793525865\\
1.68	-0.00431261184225993\\
1.8	-0.00689670655959402\\
1.92	-0.00397483159849098\\
2.04	-0.00708921907180836\\
2.16	-0.00368877736347232\\
2.28	-0.00713907615238215\\
2.4	-0.00346364239870288\\
2.52	-0.0070733905515255\\
2.64	-0.00331684937544762\\
2.76	-0.00691716283633012\\
2.88	-0.00326773236021988\\
3	-0.00669066821706902\\
};
\addlegendentry{$r=10$}

\end{axis}

\begin{axis}[%
width=2.5in,
height=2.1in,
at={(3.1in,0in)},
scale only axis,
xmin=0,
xmax=3,
ymin=-0.02,
ymax=0.0002,
axis background/.style={fill=white},
legend columns=2,
legend style={at={(0.5,0.21)},anchor=north, draw=white!15!black,text width=3.5em},
]
\addplot [color=colorclassyblue,  line width = 1.5pt]
  table[row sep=crcr]{%
0	0\\
0.12	-0.0189766781954106\\
0.24	-0.00249088466203917\\
0.36	-0.0126509801972895\\
0.48	-0.0123489919144722\\
0.6	-0.0124018803576679\\
0.72	-0.0121212716783702\\
0.84	-0.0117770655919492\\
0.96	-0.011524556321998\\
1.08	-0.0113809513132734\\
1.2	-0.0113088813017068\\
1.32	-0.0112892989221406\\
1.44	-0.0112892989221406\\
1.56	-0.011299277380006\\
1.68	-0.0113122855501061\\
1.8	-0.0113233906261872\\
1.92	-0.0113310048912945\\
2.04	-0.0113310048912945\\
2.16	-0.0113310048912945\\
2.28	-0.0113310048912945\\
2.4	-0.0113310048912945\\
2.52	-0.0113310048912945\\
2.64	-0.0113310048912945\\
2.76	-0.0113310048912945\\
2.88	-0.0113310048912945\\
3	-0.0113310048912945\\
};
\addlegendentry{$r=1$}

\addplot [only marks, mark=*, mark size = 2pt, mark options={solid,colorpurple},line width = 1.2pt]
  table[row sep=crcr]{%
0	0\\
0.12	-0.0169312929290579\\
0.24	-0.00176983957695752\\
0.36	-0.0114530247779794\\
0.48	-0.0116976504823636\\
0.6	-0.0120481924418051\\
0.72	-0.011893368709082\\
0.84	-0.0115838269042005\\
0.96	-0.0113318949734207\\
1.08	-0.0111571879887329\\
1.2	-0.0110551839151963\\
1.32	-0.0110058049675418\\
1.44	-0.0109914275564657\\
1.56	-0.0109914275564657\\
1.68	-0.0109914275564657\\
1.8	-0.0110015774838877\\
1.92	-0.0110081035546527\\
2.04	-0.0110081035546527\\
2.16	-0.0110081035546527\\
2.28	-0.0110081035546527\\
2.4	-0.0110081035546527\\
2.52	-0.0110081035546527\\
2.64	-0.0110081035546527\\
2.76	-0.0110081035546527\\
2.88	-0.0110081035546527\\
3	-0.0110081035546527\\
};
\addlegendentry{$r=2$}

\addplot [colorabs, mark=x, mark size = 3pt, mark options={solid,colorabs},  line width = 1.2pt]
  table[row sep=crcr]{%
0	0\\
0.12	-0.0131782292505986\\
0.24	-0.00191865398832686\\
0.36	-0.00697646119732965\\
0.48	-0.0090423786947601\\
0.6	-0.0100268722854385\\
0.72	-0.0103615378379212\\
0.84	-0.0103994956977546\\
0.96	-0.010317555405039\\
1.08	-0.0101611522614765\\
1.2	-0.00999215256312058\\
1.32	-0.0098857055442056\\
1.44	-0.00986120016593752\\
1.56	-0.00984879932581169\\
1.68	-0.0098110317875025\\
1.8	-0.0097578752015372\\
1.92	-0.0097578752015372\\
2.04	-0.00979096681967497\\
2.16	-0.00981458720923425\\
2.28	-0.00979450425151354\\
2.4	-0.00975765973360658\\
2.52	-0.00975765973360658\\
2.64	-0.00979005834612163\\
2.76	-0.00981183160773864\\
2.88	-0.00978997836019054\\
3	-0.00975263968414929\\
};
\addlegendentry{$r=4$}

\addplot [color=colorimag, densely dotted, line width = 1.5pt]
  table[row sep=crcr]{%
0	0\\
0.12	-0.0087135901919758\\
0.24	-0.0129076586002765\\
0.36	0.000194680475098047\\
0.48	-0.00521785365741755\\
0.6	-0.00652210543369511\\
0.72	-0.00541735085121351\\
0.84	-0.00808242260431139\\
0.96	-0.0049194023386037\\
1.08	-0.00832891274672318\\
1.2	-0.00459707115028286\\
1.32	-0.00836223834706686\\
1.44	-0.00437716160579553\\
1.56	-0.00831822745793415\\
1.68	-0.00424922358141839\\
1.8	-0.00820404165041049\\
1.92	-0.00424779776907767\\
2.04	-0.00801544161454253\\
2.16	-0.0044024994761409\\
2.28	-0.00774807825419933\\
2.4	-0.0047296487556138\\
2.52	-0.00741878415454918\\
2.64	-0.00524162438448607\\
2.76	-0.00702836317271723\\
2.88	-0.00588760370121372\\
3	-0.00659303591312346\\
};
\addlegendentry{$r=10$}

\end{axis}

\end{tikzpicture}%
}{}
\end{center}
	\caption{NLH (\ref{eq:nonlinearPDEex1}) with initial and boundary conditions (\ref{eq:datatest}). Sequence of parameters $\{\lambda_{r,n}^*\}_{n\geqslant 0}$ for CS($\lambda$), (\ref{explNLH}), obtained using Algorithm~\ref{General}.\ref{alg:opt} (left) and Algorithm~\ref{General}.\ref{alg:opt2} (right) with different values of $r$.}
	\label{fig:advs0}
\end{figure}
\begin{figure}[tbp]
\begin{center}
	\iftoggle{pgfplots}{\input{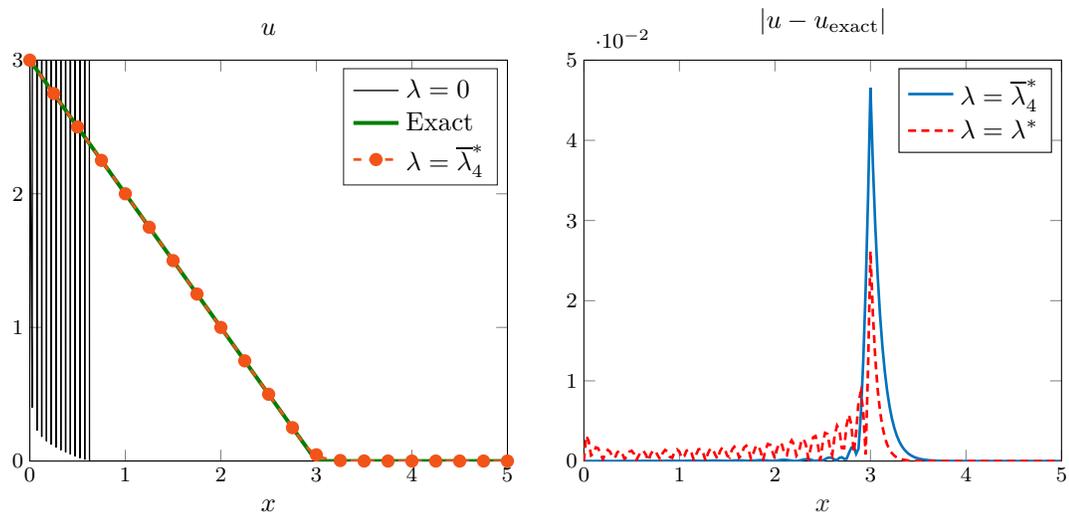}}{}
\end{center}
	\caption{NLH (\ref{eq:nonlinearPDEex1}) with initial and boundary conditions (\ref{eq:datatest}). Exact and numerical solutions of CS($\lambda$), (\ref{explNLH}), with $\lambda=0$ and $\lambda=\overline\lambda_{4}^*$ (left). Solution error for $\lambda=\overline\lambda_{4}^*$ and $\lambda=\lambda^*$ (right).}
	\label{fig:advsopt}
\end{figure}

The two algorithms introduced in this paper give equally accurate solutions. By increasing $r$, the computation time decreases, while the accuracy in the solution is only marginally affected. Moreover, both new algorithms avoid instabilities that occur when $\lambda=0$. This is shown on the left of Figure~\ref{fig:advsopt} where, as an example, we show the solution of CS($\overline\lambda^*_{4}$) and CS(0) at the final time.

On the right of Figure~\ref{fig:advsopt}, we plot the pointwise errors given by CS($\lambda^*$) and CS($\overline\lambda^*_{4}$). The error obtained with $\lambda=\overline\lambda^*_{4}$ is almost entirely located at the interface where the solution is non differentiable. Fixing $\lambda=\lambda^*$, the $L^2$ error is lower and the solution is more accurate around the interface. However, spurious oscillations are seen where the true solution is smooth.

The second benchmark problem is (\ref{eq:nonlinearPDEex1}) with
\begin{equation}\label{test2}
u(x,0)=\left(1-\frac{x^2}6\right)_+,\qquad u(-6,t)=u(6,t)=0,\qquad (x,t)\in[-6,6]\times[0,9],
\end{equation}
where $f_+=\max(f,0).$ The solution of this problem is the Barenblatt profile,
$$u_{\rm exact}(x,t)=(t+1)^{-1/3}\left(1-\frac{x^2}{6(t+1)^{2/3}}\right)_+.$$ This solution has compact support and is not differentiable at the interface points, which move outward at a finite speed. We solve this problem with $\dx=0.02$ and $\dt=0.09$.

We show in Figure~\ref{fig:advs02} that in this case the two proposed algorithms generate sequences of parameters that approach a small negative value for all the considered values of $r$.

\begin{table}[t]
\caption{NLH (\ref{eq:nonlinearPDEex1}) with initial and boundary conditions (\ref{test2}). Errors in solution and conservation laws and computation time.}\label{tab:errtab2}
\small
\centerline{\begin{tabular}{|c|c|c|c|c|}
\hline
$\lambda$ &  $\text{Err}_1$ & $\text{Err}_2$ & Solution error & Computation time\\ 
\hline
\hline
$\{\lambda_{1,n}^*\}_{n\geqslant 0}$	&  2.00e-13   & 2.55e-14 & 1.82e-04 & 8.18 \\	
\hline
$\{\lambda_{2,n}^*\}_{n\geqslant 0}$	&  1.41e-13   & 2.43e-14 & 1.41e-04 & 1.73 \\	
\hline
$\{\lambda_{4,n}^*\}_{n\geqslant 0}$	&  3.43e-13   & 3.28e-14 & 1.45e-04 & 1.51 \\
\hline
$\{\lambda_{10,n}^*\}_{n\geqslant 0}$	&  3.62e-13   & 3.20e-14 & 1.15e-04 & 1.35 \\	
\hline
\hline
$\overline\lambda_{1}^*=-6.81$e-04	&  5.05e-14   & 3.05e-14 & 5.20e-04 & 5.88 \\	
\hline
$\overline\lambda_{2}^*=-5.64$e-04	&  2.70e-14   & 2.72e-14 & 4.53e-04 & 1.64 \\	
\hline
$\overline\lambda_{4}^*=-4.38$e-04	&  2.79e-14   & 3.10e-14 & 3.78e-04 & 1.26 \\	
\hline
$\overline\lambda_{10}^*=-3.86$e-04	&  2.36e-14   & 2.93e-14 & 3.45e-04 & 1.14 \\	
\hline
\hline
$\lambda=0$	& 3.59e-14   & 2.64e-14 & 0.2989 &  0.59 \\	
\hline
$\lambda^*=-2.32$e-04	&  3.85e-14   & 3.61e-14 & 2.62e-04 & N/A\\
\hline
\end{tabular}}
\end{table}
\begin{figure}[tbp]
\begin{center}
	\iftoggle{pgfplots}{\input{L2.tex}}{}
\end{center}
	\caption{NLH (\ref{eq:nonlinearPDEex1}) with initial and boundary conditions (\ref{eq:datatest}). Sequence of parameters $\{\lambda_{r,n}^*\}_{n\geqslant 0}$ for CS($\lambda$), (\ref{explNLH}), obtained using Algorithm~\ref{General}.\ref{alg:opt} (left) and Algorithm~\ref{General}.\ref{alg:opt2} (right) with different values of $r$.}
	\label{fig:advs02}
\end{figure}
\begin{figure}[tbp]
\begin{center}
	\iftoggle{pgfplots}{\input{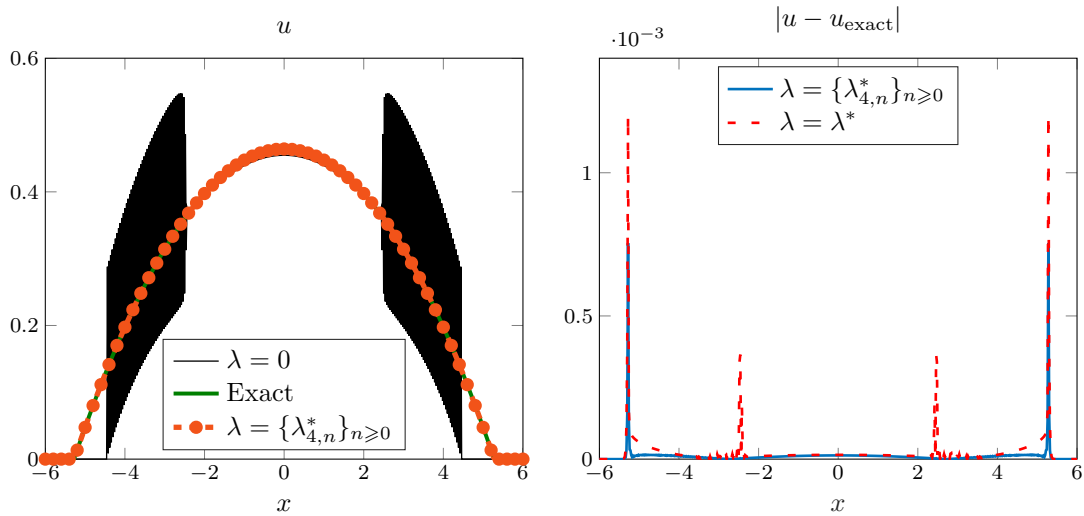}}{}
\end{center}
	\caption{NLH (\ref{eq:nonlinearPDEex1}) with initial and boundary conditions (\ref{test2}). Exact and numerical solutions of CS($\lambda$), (\ref{explNLH}), with $\lambda=0$ and $\lambda=\{\lambda_{4,n}^*\}_{n\geqslant 0}$ (left). Solution error for $\lambda=\lambda=\{\lambda_{4,n}^*\}_{n\geqslant 0}$ and $\lambda=\lambda^*$ (right).}
	\label{fig:advsopt2}
\end{figure}

Table~\ref{tab:errtab2} shows that all the schemes preserve a discrete version of the global invariants. When the value of the parameter is fixed during the iteration, this is a consequence of the preservation of the local conservation laws. When a sequence of different parameters is used, this is instead due to the zero boundary conditions (see Remark~\ref{remNLH}). Although in this case the conservation laws are not preserved locally, the error in the solution is the lowest.

The figures in Table~\ref{tab:errtab2} show that with increasing $r$ the computation time decreases while the accuracy is not significantly affected even for $r=10$. This is a consequence of the fact that the sequence of parameters obtained for  different values of $r$ quickly approach each other after just a few time-steps (see Figure~\ref{fig:advs02}). Once again, we find in Figure~\ref{fig:advsopt2} (left) that the sequence of parameters $\{\lambda_{4,n}^*\}_{n\geqslant 0}$ obtained from Algorithm~\ref{General}.\ref{alg:opt} allows us to avoid the numerical instability that occurs under the default choice, $\lambda=0$.

On the right half of Figure~\ref{fig:advsopt2} we show the solution error of CS($\lambda_{4,n}^*$)$_{n\geqslant 0}$ and of CS($\lambda^*$). The solution error obtained using the sequence of parameters given by Algorithm~\ref{General}.\ref{alg:opt} is mainly located at the points where the solution is not differentiable. Although $\lambda=\lambda^*$ minimizes the $L^2$ norm of the error with respect to any other fixed value of $\lambda$, we notice that the error at the interface can be further reduced by changing the parameter at every time-step. Moreover, for $\lambda=\lambda^*$, relatively large components of error appear near $\pm2.5$ where the solution is otherwise smooth. These are not visible in the solution errors obtained using either the adaptive sequences given by Algorithm~\ref{General}.\ref{alg:opt} or the fixed parameters given by Algorithm~\ref{General}.\ref{alg:opt2}.

\section{Conclusions and future perspectives}
In this paper we have proposed two approaches for identifying an optimal method in a parameter dependent family of numerical schemes, based on a minimization of the defect as an estimate of the local error. The first approach uses different (adaptive) values of the parameters at every time-step. In the second approach fixed values of the parameters are derived from a sequence. The latter approach does not compromise parameter depending conservation properties of geometric integrators.

The new algorithms solve an optimization problem at each time-step in order to identify the optimal values of the parameter. In principle, this can increase the computational cost of the original method prohibitively. However, in the large time-step regime, it is possible to solve the optimization problem on coarser spatial grids without compromising the accuracy of the optimal parameters, significantly decreasing the computational time.

The new approaches have been applied to families of schemes for the KdV equation and a nonlinear heat equation that preserve local conservation laws. The proposed numerical tests show that, on one hand, the new strategies effectively identify very accurate methods in each considered family of schemes. On the other hand, introducing a coarse grid for the solution of the optimization problem tremendously improves the efficiency of the new strategies. Overall, the computational time is comparable to that of other schemes in literature, while the accuracy of the proposed approach is much superior.

\subsection*{Acknowledgments}
The authors would like to thank the Isaac Newton Institute for Mathematical Sciences for support and hospitality during the programme Geometry, Compatibility and Structure Preservation in Computational Differential Equations, when work on this paper was undertaken. This work was supported by EPSRC grant number EP/R014604/1. 

\bibliographystyle{acm}
\bibliography{references}

\end{document}